\documentclass [11pt]{amsart}

\usepackage{amssymb,amsmath,amscd,amsthm}

\newtheorem{lemma}{Lemma}
\newtheorem{proposition}{Proposition}
\newtheorem{definition}{Definition}
\newtheorem{corollary}{Corollary}

\newcommand{\tr}{\,\mathrm{tr}\,}
\newcommand{\ad}{\,\mathrm{ad}\,}
\begin{document}

\begin{center}

{\Large {\bf Automorphisms of Chevalley groups \\

\bigskip

of types $B_2$ and $G_2$ over local rings\footnote{ he work is
supported by the Russian President grant MK-904.2006.1 and by the
grant of Russian Fond of Basic Research 05-01-01048.} }}

\bigskip
\bigskip

{\large \bf E.I~.~Bunina}

\end{center}
\bigskip

\begin{center}

{\bf Abstract.}

\end{center}

In the paper we prove that every automorphism of any adjoint
Chevalley group of types $B_2$ or  $G_2$ is standard, i.e., it is a
composition of the ``inner'' automorphism, ring automorphism and
central automorphism.

\bigskip

\section*{Introduction.}\leavevmode

An associative ring $R$ with a unit is called \emph{local}, if it
has exactly one maximal ideal (which coincides with the Jacobson
radical of~$R$). Equivalently, all non-invertible elements of~$R$
form an ideal. In this paper all rings under consideration are
commutative.

Let $G_{ad}$ be a Chevalley-Demazure group scheme associated with an
irreducible root system~$\Phi$ of type $B_2$ or $G_2$ (see detailed
definitions in the next section); $G_{ad}(R,\Phi)$ be a set of
points~$G$ with values in~$R$; $E_{ad}(R,\Phi)$ be the elementary
subgroup of~$G_{ad}(R,\Phi)$, where $R$ is a local commutative ring
with~$1$. In this paper we describe automorphisms of
$G_{ad}(R,\Phi)$ and $E_{ad}(R,\Phi)$ (for the root systems $A_l,
D_l $ and $E_l$ the automorphisms were described in the
paper~\cite{ravnyekorni}). Suppose that
 $R$ is a local ring with~$1/2$ and $1/3$. Then every automorphism of $G_{ad}(R,\Phi)$
($E_{ad}(R,\Phi)$) is standard (see below definitions of standard
automorphisms). These results for Chevalley groups over fields were
proved by R.\,Steinberg~\cite{Stb1} for finite case and by
J.\,E.\,Humphreys~\cite{H} for infinite case. K.\,Suzuki~\cite{Su}
studied automorphisms of Chevalley groups over rings of $p$-adic
numbers. E.\,Abe~\cite{Abe_OSN} proved this result for Noetherian
rings, but the class of all local rings is not completely contained
in the class of Noetherian rings, and the proof of~\cite{Abe_OSN}
can not be extended to the case of arbitrary local rings.

From the other side, automorphisms of classical groups over  rings
were discussed by many authors. This field of research was open by
Schreier and Van der Varden, who described automorphisms of the
group $PSL_n$ ($n\ge 3$) over arbitrary field. Then
J.\,Diedonne~\cite{D}, L. Hua and  I.\,Reiner~\cite{HR},
O'Meara~\cite{O'M2}, B.\,R.\,McDonald~\cite{Mc}, I.\,Z.\,Golubchik
and A.\,V.\,Mikhalev~\cite{GolMikh1} and others studied this problem
for groups over more general rings. To prove our theorem we
generalize some methods from the paper of
V.\,M.\,Pe\-te\-chuk~\cite{Petechuk1}.

Every Chevalley group under consideration is embedded into the group
$GL_N(R)$ for some~$N\in \mathbb N$. Therefore we can consider
Chevalley groups as matrix groups and use linear algebraic group
technique: invertible coordinate changes in local rings, uniqueness
of a solution of systems of linear equations over local rings with
the condition, that determinant of a corresponding matrix is
invertible, and so on. As the result we come to the fact that every
automorphism of Chevalley group is induced by automorphism of the
corresponding matrix ring.

The author would like to thank A.Yu.~Golubkov, A.A.~Klyachko and
A.V.~Mikhalev for valuable comments and attention to the work.

\section{Definitions and formulations of main theorems.}\leavevmode

\subsection{Root systems.}

 \begin{definition}
 \emph{A finite nonempty set $\Phi\subset \mathbb R^l$ of vectors of the space~$\mathbb R^l$ is called
 a \emph{root system}, if it generates $\mathbb R^l$, does not contain~$0$
and satisfies the following properties:}

\emph{1) $\forall \alpha\in \Phi\ (c\cdot \alpha  \in
\Phi\Leftrightarrow c=\pm 1)$;}

\emph{2) if we introduce
$$
\langle \alpha,\beta\rangle:=
\frac{2(\alpha,\beta)}{(\alpha,\alpha)} \quad
\text{(\emph{reflection coefficient})}
$$
for $\alpha,\beta\in \mathbb R^l$, then for any $\alpha,\beta\in
\Phi$ we have $\langle \alpha,\beta\rangle \in \mathbb Z$;}

\emph{3) let for $\alpha\in \mathbb R^l$, $w_\alpha$ be a reflection
under a hyperplane, orthogonal to the vector~$\alpha$, i.\,e.,
$\forall \beta \in \mathbb R^l$
$$
w_\alpha (\beta)=\beta-\langle \alpha,\beta\rangle \alpha.
$$
Then for any $\alpha,\beta\in \Phi$ we have $w_\alpha(\beta)\in
\Phi$, i.\,e., the set $\Phi$ is invariant under the action of all
reflections~$w_\alpha$, $\alpha\in \Phi$.}
\end{definition}

If $\Phi$ is a root system, then its elements are called
\emph{roots}.

\begin{definition}
\emph{The group $W$ generated by all reflections $w_\alpha$,
$\alpha\in \Phi$, is called a \emph{Weil group} of the
system~$\Phi$.}
\end{definition}

\begin{definition}
\emph{If we put in the space $\mathbb R^l$ a hyperplane, that does
not contain any roots from~$\Phi$, then all roots are divided into
two disjoint sets of \emph{positive} ($\Phi^+$) and \emph{negative}
($\Phi^-$) roots.
 A \emph{system of simple roots} is a set $\Delta= \{
\alpha_1,\dots,\alpha_l\}\subset \Phi^+$ such that any positive root
$\beta\in \Phi^+$ is uniquely written in the form
$n_1\alpha_1+\dots+n_l\alpha_l$, where $n_1,\dots,n_l\in \mathbb
Z^+$.}
\end{definition}

For every root system $\Phi$ there exists a system of simple roots.
The number $l$ is called a \emph{rank} of the system~$\Phi$.

Among all root systems we are interested in \emph{undecomposable
systems}, i.\,e., such  systems~$\Phi$ that can not be represented
as a union $\Phi=\Phi_1\cup \Phi_2$ of two disjoint sets with
mutually orthogonal roots.

\begin{definition}
\emph{By any root system one can construct the following
\emph{Dynkin diagram}. It is a graph that is constructed as follows:
its vertices correspond to the simple roots
$\alpha_1,\dots,\alpha_l$, two vertices with numbers $i$ and~$j$ are
connected, if $\langle \alpha_i,\alpha_l\rangle\ne 0$. If
$|\alpha_i|=|\alpha_j|$, then $\langle \alpha_i,\alpha_j\rangle=
\langle \alpha_j,\alpha_i\rangle$ and the number of edges between
vertices $i$ and~$j$ is equal to $|\langle
\alpha_i,\alpha_j\rangle|$. If $|\alpha_i|> |\alpha_j|$, and
$\langle \alpha_i,\alpha_j\rangle< \langle \alpha_j,\alpha_i\rangle$
and $|\langle \alpha_i,\alpha_j\rangle|=1$. In this case the
vertices $i$ and~$j$ are connected by $|\langle
\alpha_j,\alpha_i\rangle|$ edges and an arrow comes from a long root
to a short one.}
\end{definition}

By a Dynkin diagram we can uniquely define a root system.

All undecomposable root systems up to an isomorphism are divided
into $4$ infinite (\emph{classical}) series $A_l$ ($l\ge 1$), $B_l$
($l\ge 2$), $C_l$ ($l\ge 3$) and $D_l$ ($l\ge 4$) and $5$ separate
(\emph{exceptional}) cases $E_6$, $E_7$, $E_8$, $F_4$ and~$G_2$.

In this paper we are interested in the root systems $B_2$ and $G_2$,
with Dynkin diagrams

\bigskip

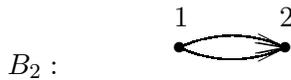
\begin{picture}(250,40)
\put(5,20){$B_2:$} \put(70,30){\circle*{4}}
\put(110,30){\circle*{4}} \put(68,38){$1$} \put(108,38){$2$}
\qbezier(70,30)(90,39)(110,30) \qbezier(70,30)(90,21)(110,30)
\put(110,30){\line(-2,1){10}} \put(110,30){\line(-2,-1){10}}
\put(110,30){\line(-6,1){11}} \put(110,30){\line(-5,-1){11}}
\end{picture}

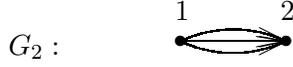
\begin{picture}(250,40)
\put(5,25){$G_2:$} \put(70,30){\circle*{4}}
\put(110,30){\circle*{4}} \put(70,30){\line(1,0){40}}
\put(68,38){$1$} \put(108,38){$2$} \qbezier(70,30)(90,39)(110,30)
\qbezier(70,30)(90,21)(110,30) \put(110,30){\line(-2,1){10}}
\put(110,30){\line(-2,-1){10}} \put(110,30){\line(-6,1){11}}
\put(110,30){\line(-5,-1){11}}
\end{picture}

Here we fix a root system~$\Phi$, of type $B_2$ or $G_2$, with
system of simple roots~$\Delta(B_2)=\{ \alpha_1=e_1-e_2,
\alpha_2=e_2\}$, or $\Delta(G_2)=\{ \alpha_1= e_1-e_2,
\alpha_2=-2e_1+e_2+e_3\}$, positive roots $\Phi^+(B_2)=\{ \alpha_1,
\alpha_2, \alpha_1+\alpha_2=e_1, \alpha_1+2\alpha_2=e_1+e_2\}$, or
$\Phi^+(G_2)=\{ \alpha_1, \alpha_2, \alpha_1+\alpha_2= e_3-e_1,
2\alpha_1+\alpha_2=e_3-e_2, 3\alpha_1+\alpha_2=e_1+e_3-2e_2,
3\alpha_1+2\alpha_2=2e_3-e_1-e_2\} $, Weil group~$W$. Recall that in
our case in the root system there are roots of two lengths, all
roots of the same length are conjugate by the action of the Weil
group.  More details about root systems and their properties can be
found in the books \cite{Hamfris}, \cite{Burbaki}.

\subsection{Semisimple Lie algebras.}

More details about semisimple Lie algebras can be found in the
book~\cite{Hamfris}.

\begin{definition}
\emph{\emph{Lie algebra} $\mathcal L$ over a field~$K$ is a linear
space over~$K$, with an operation of multiplication $x,y\mapsto
[x,y],$ linear by both variables and satisfying the following
conditions:}

\emph{1) \emph{anticommutativity}:
$$
\forall x,y\in {\mathcal L}\ [x,y]=-[y,x];
$$}

\emph{2) \emph{Jacobi identity}:
$$
\forall x,y,z\in {\mathcal L}\ [x,[y,z]]+[y,[z,x]]+[z,[x,y]]=0.
$$}
\end{definition}

Dimension of Lie algebra is defined by its dimension as a linear
space over~$K$. So Lie algebra is called \emph{finite dimensional},
if the space $\mathcal L$ is finite dimensional.

\begin{definition}
\emph{A subspace  ${\mathcal L}'$ of a Lie algebra~$\mathcal L$ as a
linear space is called  a \emph{subalgebra} of~$\mathcal L$, if
$\forall x,y\in {\mathcal L}'$ $[x,y]\in {\mathcal L}'$. A
subalgebra ${\mathcal L}'$ of~$\mathcal L$ is called its
\emph{ideal}, if $\forall x\in {\mathcal L}'$ $\forall y\in
{\mathcal L}$\ $[x,y]\in {\mathcal L}'$.}  \emph{A \emph{commutant}
of Lie algebras ${\mathcal L}_1$ and ${\mathcal L}_2$, lying in one
Lie algebra~$\mathcal L$, is its subalgebra ${\mathcal
L}'=[{\mathcal L}_1,{\mathcal L}_2],$ generated by all $[x,y]$ for
$x\in {\mathcal L}_1$, $y\in {\mathcal L}_2$.}
\end{definition}

\begin{definition}
\emph{A sequence
$$
{\mathcal L}^0={\mathcal L}, {\mathcal L}^1=[{\mathcal L},{\mathcal
L}^0], {\mathcal L}^2=[ {\mathcal L},{\mathcal L}^1],\dots,
{\mathcal L}^{n+1}=[{\mathcal L},{\mathcal L}^n],\dots
$$
is called a \emph{central series} of a Lie algebra~${\mathcal L}$.
If for some $n\in \mathbb N$ we have ${\mathcal L}^n=0$, then a Lie
algebra~$\mathcal L$ is called \emph{nilpotent}.}
\end{definition}

\begin{definition}
\emph{A sequence
$$
{\mathcal L}^{(0)}={\mathcal L}, {\mathcal L}^{(1)}=[{\mathcal
L}^{(0)},{\mathcal L}^{(0)}], {\mathcal L}^{(2)}=[{\mathcal
L}^{(1)},{\mathcal L}^{(1)}],\dots, {\mathcal L}^{(n+1)}=[{\mathcal
L}^{(n)},{\mathcal L}^{(n)}],\dots
$$
is called a \emph{derivative series} of a Lie algebra~$\mathcal L$.
If for some $n\in \mathbb N$ we have ${\mathcal L}^{(n)}=0$, then
$\mathcal L$ is called \emph{solvable}.}
\end{definition}

\begin{definition}
\emph{Consider a finitely dimensional Lie algebra~$\mathcal L$. The
greatest solvable ideal of Lie algebra, that is a sum of all its
solvable ideals, is called a \emph{radical} of Lie algebra. A Lie
algebra with zero radical is called \emph{semisimple}. A
noncommutative Lie algebra $\mathcal L$ is called \emph{simple}, if
its has exactly two ideals: $0$ and $\mathcal L$.}
\end{definition}

Finitely dimensional semisimple Lie algebra over $\mathbb C$ is a
direct sum of simple Lie algebras.

\begin{definition}
\emph{A \emph{normalizer} of a subalgebra ${\mathcal L}'$ in an
algebra~$\mathcal L$ is a subalgebra
$$
N_{\mathcal L}({\mathcal L}'):=\{ x\in {\mathcal L}\,|\, \forall
y\in {\mathcal L}' \ [x,y]\in {\mathcal L}'\}.
$$}
\end{definition}

\begin{definition}
\emph{A \emph{Cartan subalgebra} of a Lie algebra~${\mathcal L}$ is
its nilpotent self-normalizing subalgebra~$\mathcal H$. For a
semisimple Lie algebra it is Abelian and is defined up to an
automorphism of the algebra.}
\end{definition}

\begin{proposition}\emph{(\cite{Hamfris}, \S\,8)}
Let $\mathcal L$ be a semisimple finitely dimensional Lie algebra
over~$ \mathbb C$, $\mathcal H$ its Cartan subalgebra. Consider the
space~${\mathcal H}^*$. Let for $\alpha\in {\mathcal H}^*$
$$
{\mathcal L}_\alpha:=\{ x\in {\mathcal L}\mid [h,x]=\alpha(h)x\text{
for any } h\in {\mathcal H}\}.
$$

In this case  ${\mathcal L}_0={\mathcal H}$ and the algebra
$\mathcal L$ allows a decomposition ${\mathcal L}={\mathcal H}
\oplus \sum\limits_{\alpha\ne 0} {\mathcal L}_\alpha$, and if
${\mathcal L}_\alpha\ne 0$, then $\dim {\mathcal L}_\alpha=1$, all
such nonzero $\alpha\in {\mathcal H}$ that ${\mathcal L}_\alpha\ne
0$, form some root system~$\Phi$. A root system $\Phi$ and a
semisimple Lie algebra $\mathcal L$ over~$\mathbb C$ uniquely define
each other.
\end{proposition}

On a Lie algebra $\mathcal L$ one can introduce a bilinear
\emph{Killing form}
$$
\varkappa(x,y)=\tr (\ad x\ad y),
$$
where $\ad x\in GL({\mathcal L})$, $\ad x: z\mapsto [x,z]$ is an
\emph{adjoint representation}. For a semisimple Lie algebra a
restriction of the Killing form on~$\mathcal H$ is non-degenerate,
so we can identify the spaces $\mathcal H$ and ${\mathcal H}^*$.

\begin{proposition}\emph{(\cite{Steinberg}, p.\,10)}
There exists a basis $\{ h_1, \dots, h_l\}$ in~$\mathcal H$ and for
every $\alpha\in \Phi$ elements $x_\alpha \in {\mathcal L}_\alpha$
so that

\emph{1)} $\{ h_i; x_\alpha\}$ is a basis in~$\mathcal L$;

\emph{2)} $[h_i,h_j]=0$;

\emph{3)} $[h_i,x_\alpha]=\langle \alpha,\alpha_i\rangle x_\alpha$;

\emph{4)} $[x_\alpha,x_{-\alpha}]=h_\alpha=$ is an integral linear
combination of~$h_i$;

\emph{5)} $[x_\alpha,x_\beta]=N_{\alpha \beta} x_{\alpha+\beta}$,
если $\alpha+\beta\in \Phi$ ($N_{\alpha\beta}\in \mathbb Z$);

\emph{6)} $[x_\alpha,x_\beta]=0$, if $\alpha+\beta\ne 0$,
$\alpha+\beta\notin \Phi$.
\end{proposition}

\begin{definition}
\emph{A \emph{representation} of Lie algebra~$\mathcal L$ in a
linear space~$V$ is a linear mapping $\pi: {\mathcal L}\to gl(V)$,
with
$$
\forall x,y\in {\mathcal L}\ \pi ([x,y])=\pi(x)\pi(y)-\pi(y)\pi(x).
$$
A representation is called \emph{faithful}, of it has zero kernel.}
\end{definition}

\subsection{Elementary Chevalley groups.}

Introduce now elementary Chevalley groups (see~\cite{Steinberg}).

Let  $\mathcal L$ be a semisimple Lie algebra (over~$\mathbb C$)
with the root system~$\Phi$, $\pi: {\mathcal L}\to gl(V)$ be its
finitely dimensional faithful representation (of dimension~$n$).
Then we can choose a basis in the space~$V$ (for example, we can
take a basis of weight vectors) so that all operators
$\pi(x_\alpha)^k/k!$ for $k\in \mathbb N$ are written as integral
(nilpotent) matrices. An integral matrix can be also considered as a
matrix over an arbitrary commutative ring with a unit. Let $R$ be  a
ring of this type. Consider the $n\times n$ matrices over~$R$, the
matrices $\pi(x_\alpha)^k/k!$ for
 $\alpha\in \Phi$, $k\in \mathbb N$ are included in $M_n(R)$.

Now consider automorphisms of the free module $R^n$ of the form
$$
\exp (tx_\alpha)=x_\alpha(t)=1+t\pi(x_\alpha)+t^2
\pi(x_\alpha)^2/2+\dots+ t^k \pi(x_\alpha)^k/k!+\dots
$$
Since all matrices $\pi(x_\alpha)$ are nilpotent, we have that this
series is finite.

\begin{definition}
\emph{The subgroup of $Aut(R^n)$, generated by all automorphisms
$x_\alpha(t)$, $\alpha\in \Phi$, $t\in R$, is called an
\emph{elementary Chevalley group} (notation:
 $E_\pi(\Phi,R)$).}
\end{definition}

\begin{definition}
\emph{In elementary Chevalley group we can introduce the following
important elements and subgroups:}

--- $w_\alpha(t)=x_\alpha(t) x_{-\alpha}(-t^{-1})x_\alpha(t)$, $\alpha\in \Phi$,
$t\in R^*$;

--- $h_\alpha (t) = w_\alpha(t) w_\alpha(1)^{-1}$;

\emph{--- the subgroup $N$ is generated by all
 $w_\alpha (t)$, $\alpha \in \Phi$, $t\in R^*$;}

\emph{--- the subgroup $H$ is generated by all
 $h_\alpha(t)$, $\alpha \in \Phi$, $t\in R^*$.}
\end{definition}

It is known that the group $N$ is a normalizer of~$H$ in elementary
Chevalley group, the quotient group $N/H$ is isomorphic to the Weil
group $W(\Phi)$.

Elementary Chevalley groups are defined not by representation of the
Chevalley groups, but just by its \emph{weight lattice}:

\begin{definition}
\emph{If $V$ is the representation space of the semisimple Lie
algebra~$\mathcal L$ (with the Cartan subalgebra~$\mathcal H$), then
a functional
 $\lambda \in {\mathcal H}^*$ is called
the \emph{weight} of its representation, if there exists a nonzero
vector $v\in V$  (which is called the \emph{weight vector}) such
that for every $h\in {\mathcal H}$ $\pi(h) v=\lambda (h)v.$ All
weights of a given representation (by addition) generate a lattice
(free Abelian group, where every  $\mathbb Z$-basis  is also a
$\mathbb C$-basis in~${\mathcal H}^*$), that is called the
\emph{weight lattice} $\Lambda_\pi$.}
\end{definition}

An elementary Chevalley group is completely defined by a root
system~$\Phi$, commutative ring~$R$ with a unit and a weight lattice
$\Lambda_\pi$.

Among all lattices we can mark two: the lattice corresponding to the
adjoint representation, it is generated by all roots (the
\emph{adjoint lattice}~$\Lambda_{ad}$) and the lattice generated by
all weights of all representations (the \emph{universal
lattice}~$\Lambda_{sc}$). For every faithful representation~$\pi$ we
have the inclusion $\Lambda_{ad}\subseteq \Lambda_\pi \subseteq
\Lambda_{sc}.$ Respectively, we have the \emph{adjoint} and
\emph{universal} elementary Chevalley groups.

\begin{proposition}\emph{(\cite{Steinberg}, p.\,32)}
Every elementary Chevalley group satisfies the following
conditions:

(R1) $\forall \alpha\in \Phi$ $\forall t,u\in R$\quad
$x_\alpha(t)x_\alpha(u)= x_\alpha(t+u)$;

(R2) $\forall \alpha,\beta\in \Phi$ $\forall t,u\in R$\quad
 $\alpha+\beta\ne 0\Rightarrow$
$$
[x_\alpha(t),x_\beta(u)]=x_\alpha(t)x_\beta(u)x_\alpha(-t)x_\beta(-u)=
\prod x_{i\alpha+j\beta} (c_{ij}t^iu^j),
$$
where $i,j$ are integers, multiplication is taken by all roots
$i\alpha+j\beta$, permuted in some fixed order; $c_{ij}$ are integer
numbers not depending of $t$ and~$u$;

(R3) $\forall \alpha \in \Phi$ $w_\alpha=w_\alpha(1)$;

(R4) $\forall \alpha,\beta \in \Phi$ $\forall t\in R^*$ $w_\alpha
h_\beta(t)w_\alpha^{-1}=h_{w_\alpha (\beta)}(t)$;

(R5) $\forall \alpha,\beta\in \Phi$ $\forall t\in R^*$ $w_\alpha
x_\beta(t)w_\alpha^{-1}=x_{w_\alpha(\beta)} (ct)$, where
$c=c(\alpha,\beta)= \pm 1$;

(R6) $\forall \alpha,\beta\in \Phi$ $\forall t\in R^*$ $\forall u\in
R$ $h_\alpha (t)x_\beta(u)h_\alpha(t)^{-1}=x_\beta(t^{\langle
\beta,\alpha \rangle} u)$.
\end{proposition}

By $X_\alpha$ we denote the subgroup generated by all $x_\alpha (t)$
for $t\in R$.

\subsection{Chevalley groups.}

Introduce now Chevalley groups (see~\cite{Steinberg},
\cite{Chevalley}, \cite{Artem_dis}, \cite{VavPlotk1}, and the later
references in these papers).

\begin{definition}\emph{A subset $X\subseteq \mathbb C^n$ is called an \emph{affine variety},
if $X$ is a set of common zeros in $\mathbb C^n$ of a finite system
of polynomials from $\mathbb C[x_1,\dots,x_n]$.}
\end{definition}

\begin{definition}
\emph{Topology of an affine $n$-space, where a system of closed sets
coincides with the system of affine varieties, is called a
\emph{Zarisski topology}. A variety is called \emph{irreducible}, if
it can not be represented as a union of two proper nonempty closed
subsets.}
\end{definition}

\begin{definition}
\emph{Let $G$ be an affine variety with a group structure. If both
mappings
\begin{align*}
m:& G\times G\to G,& m(x,y)&=xy,\\
i:& G\to G,& i(x)&=x^{-1}
\end{align*}
are expressed as polynomials of coordinates, then $G$ is called an
\emph{algebraic group}. A \emph{linear algebraic group} is an
arbitrary algebraic subgroup in $M_n(\mathbb C)$ (with matrix
multiplication).}

\emph{An algebraic group is called \emph{connected}, if it is
irreducible as a variety.}

\emph{Every algebraic group $G$ contains the unique greatest
connected solvable normal subgroup: a \emph{radical~$R(G)$}. A
connected algebraic group with trivial radical is called
\emph{semisimple}.}
\end{definition}

Consider semisimple linear algebraic groups over algebraically
closed fields. These are precisely elementary Chevalley groups
$E_\pi(\Phi,K)$ (see~\cite{Steinberg},~\S\,5).

All these groups are defined in $SL_n(K)$ as  common set of zeros of
polynomials of matrix entries $a_{ij}$ with integer coefficients
 (for example,
in the case of the root system $C_l$ and the universal
representation we have $n=2l$ and the polynomials from the condition
$(a_{ij})Q(a_{ji})-Q=0$). It is clear now that multiplication and
inversion are also defined by polynomials with integer coefficients.
Therefore, these polynomials can be considered as polynomial over
arbitrary commutative ring with a unit. Let some elementary
Chevalley group $E$ over~$\mathbb C$ be defined in $SL_n(\mathbb C)$
by polynomials $p_1(a_{ij}),\dots, p_m(a_{ij})$. For a commutative
ring~$R$ with a unit let us consider the group
$$
G(R)=\{ (a_{ij}\in M_n(R)| \widetilde p_1(a_{ij})=0,\dots
,\widetilde p_m(a_{ij})=0\},
$$
where  $\widetilde p_1(\dots),\dots \widetilde p_m(\dots)$ are
polynomials having the same coefficients as
$p_1(\dots),\dots,p_m(\dots)$, but considered over~$R$.

\begin{definition}
\emph{A group described above is called a \emph{Chevalley group}
$G_\pi(\Phi,R)$ of type~$\Phi$ over ring~$R$, for every
algebraically closed field~$K$ it coincides with elementary
Chevalley group.}
\end{definition}

\begin{definition}
\emph{The subgroup of diagonal (in the standard basis of weight
vectors) matrices of the Chevalley group $G_\pi(\Phi,R)$ is called
the  \emph{standard maximal torus} of $G_\pi(\Phi,R)$ and it is
denoted by $T_\pi(\Phi,R)$. This group is isomorphic to
$Hom(\Lambda_\pi, R^*)$.}
\end{definition}

Let us denote by $h(\chi)$ the elements of the torus $T_\pi
(\Phi,R)$, corresponding to the homomorphism $\chi\in Hom
(\Lambda(\pi),R^*)$.

In particular, $h_\alpha(u)=h(\chi_{\alpha,u})$ ($u\in R^*$, $\alpha
\in \Phi$), where
$$
\chi_{\alpha,u}: \lambda\mapsto u^{\langle
\lambda,\alpha\rangle}\quad (\lambda\in \Lambda_\pi).
$$

Note that the condition
$$
G_\pi (\Phi,R)=E_\pi (\Phi,R)
$$
is not true even for fields, that are not algebraically closed. But
if  $G$ is a universal group and the ring $R$ is \emph{semilocal}
(i.e. it contains only finite number of maximal ideals), then we
have the condition $G_{sc}(\Phi,R)=E_{sc}(\Phi,R).$ \cite{M},
\cite{Abe1}, \cite{St3}, \cite{AS}.

Let us show the difference between Chevalley groups and their
elementary subgroups in the case when a ring~$R$ is semilocal and a
Chevalley group is not universal. In this case $G_\pi
(\Phi,R)=E_\pi(\Phi,R)T_\pi(\Phi,R)$ (see~\cite{M}, \cite{Abe1},
\cite{AS}), and the elements $h(\chi)$ are connected with elementary
generators by the formula
\begin{equation}\label{ee4}
h(\chi)x_\beta (\xi)h(\chi)^{-1}=x_\beta (\chi(\beta)\xi).
\end{equation}

It is known that the group of elementary matrices
$E_2(R)=E_{sc}(A_1,R)$ is not necessary normal in the special linear
group $SL_2(R)=G_{sc}(A_1,R)$ (see~\cite{Cn}, \cite{Sw},
\cite{Su1}).

But if $\Phi$ is an irreducible root system of the rank $l\ge 2$,
then $E(\Phi,R)$ is always normal in $G(\Phi,R)$. In the case of
semilocal rings from the formula~\eqref{ee4} we see that
$$
[G(\Phi,R),G(\Phi,R)]\subseteq E(\Phi,R).
$$
If the ring $R$ also contains~$1/2$, then it is easy to show that
$$
[G(\Phi,R),G(\Phi,R)]=[E(\Phi,R), E(\Phi,R)]=E(\Phi,R).
$$

\subsection{Definitions of standard automorphisms and formulations of main theorems.}

\begin{definition}
\emph{Let us define three types of automorphisms of the Chevalley
group
 $G_\pi(\Phi,R)$, that
are call \emph{standard}.}

 \emph{{\bf Central automorphisms.} Let $C_G(R)$ be the center of
$G_\pi(\Phi,R)$ and $\tau: G_\pi(\Phi,R) \to C_G(R)$ is a
homomorphism of groups. Then the mapping
 $x\mapsto \tau(x)x$ from $G_\pi(\Phi,R)$ onto itself is an
automorphism of  $G_\pi(\Phi,R)$, that is denoted by the
letter~$\tau$ and called a \emph{central automorphism}
of~$G_\pi(\Phi,R)$.}

\emph{Every central automorphism of the Chevalley group
$G_\pi(\Phi,R)$ is identical on its commutant. By our assumptions
the elementary subgroup $E_\pi (\Phi,R)$ is a commutant of
$G_\pi(\Phi,R)$ and $E_\pi(\Phi,R)$, therefore on elementary
Chevalley groups all central automorphism are identical. }

\emph{{\bf Ring automorphisms.} Let $\rho: R\to R$ be an
automorphism of the ring~$R$. The mapping $x\mapsto \rho \circ x$
from $G_\pi(\Phi,R)$ onto itself is an automorphism of
$G_\pi(\Phi,R)$, that is denoted by the same letter~$\rho$ and is
called the \emph{ring automorphism} of~$G_\pi(\Phi,R)$. Note that
for all $\alpha\in \Phi$ and $t\in R$ an element $x_\alpha(t)$ is
mapped into $x_\alpha(\rho(t))$.}

\emph{{\bf ``Inner'' automorphisms.} Let $V$ be the space of the
representation~$\pi$ of the group
 $G_\pi (\Phi,R)$, $g\in GL(V)$ is such a matrix that
 $g G_\pi (\Phi,R) g^{-1}=
G_\pi (\Phi,R)$. Then the mapping $x\mapsto gxg^{-1}$ from
$G_\pi(\Phi,R)$ onto itself is an automorphism of~$G(R)$, that is
denoted by $i_g$ and is called the \emph{``inner'' automorphism}
of~$G(R)$, \emph{induced by the element}~$g$ of the group~$GL(V)$.}

\emph{Similarly we can define three types of automorphisms of the
elementary subgroup~$E(R)$. An automorphism~$\sigma$ of
 $G_\pi(\Phi,R)$ (or $E_\pi(\Phi,R)$)
is called \emph{standard}, if it a composition of automorphisms of
these three types.}
\end{definition}

{\bf  Theorem 1.} \emph{Let $E_{ad}(\Phi,R)$ be an elementary
Chevalley group with an irreducible root system of  types $B_2$ or
$G_2$, $R$ be a commutative local ring with~$1/2$. Suppose that if
$\Phi=G_2$ then $1/3\in R$. Then every automorphism of
$E_{ad}(\Phi,R)$ is standard.}

In our case we have the same theorem for the Chevalley groups
$G_{ad}(\Phi,R)$:

\bigskip

{\bf  Theorem 2.} \emph{Let $G_{ad}(\Phi,R)$ be a Chevalley group
with an irreducible root system of  types $B_2$ or $G_2$, $R$ be a
commutative local ring with~$1/2$. Suppose that if $\Phi=G_2$ then
$1/3\in R$. Then every automorphism of $G_{ad}(\Phi,R)$ is
standard.}

\bigskip

Description of nonstandard automorphisms of the groups $SL_3(R)$,
$GL_3(R)$ over local rings with noninvertible~$2$ can be found
in~\cite{Petechuk2}.

The next sections are devoted to the proof of the main theorems.

\section{Replacing the initial automorphism with the special one.}\leavevmode

In this section we use some reasonings from~\cite{Petechuk1}.

\begin{definition}
\emph{By $GL_n(R,J)$ we denote the group of  matrices~$A$ from
$GL_n(R)$ such that $A-E\in M_n(J)$, $J$ is the radical of~$R$.}
\end{definition}

\begin{proposition}\label{p1_0}
By an arbitrary automorphism $\varphi$ of elementary Chevalley group
$E_{ad}(\Phi,R)$ one can construct an isomorpism $\varphi'=
i_{g^{-1}} \varphi$, $g\in GL_n(R)$ of the group
 $E_{ad}(\Phi,R)\subset GL_n(R)$ onto some subgroup in $GL_n(R)$,
such that every matrix $A\in E_{ad}(\Phi,R)$ with elements from the
subring of~$R$, generated by unit, is mapped under the action of
this isomorphism~$\varphi'$ to some matrix from the set~$A\cdot
GL_n(R,J)$.
\end{proposition}
\begin{proof}
Let $J$ be the maximal ideal (radical) of~$R$, $k$ the residue field
$R/J$. Then $E_J=E( \Phi,R,J)$ is a group generated by all
$x_{\alpha}(t)$, $\alpha\in \Phi$, $t\in J$, is the greatest normal
proper subgroup in $E(\Phi,R)$ (see~\cite{Abe1}). Therefore, $E_J$
is invariant under the action of~$\varphi$.

By this reason
$$
\varphi: E (\Phi,R)\to E(\Phi,R)
$$
induces an automorphism
$$ \overline \varphi: E (\Phi,R)/E_J=E
(\Phi,k)\to E(\Phi,k).
$$
The group $E(\Phi,k)$ is a Chevalley group over field, therefore
automorphism $\overline \varphi$ is standard, i.\,e., it has a form
$$
\overline \varphi =  i_{\overline g} \overline \rho,\quad \overline
g\in E_J(E(\Phi,k))\quad \text{ (see~\cite{Steinberg}, \S\,10).}
$$
It is clear that there exists a matrix $g\in GL_n(R)$ such that it
image under factorization   $R$ by~$J$ is~$\overline g$. Note that
it is not necessarily  $g\in N(E(\Phi,R))$.

Consider the mapping $\varphi'= i_{g^{-1}} \varphi$. it is an
isomorphism from the group
 $E_{ad}(\Phi,R)\subset GL_n(R)$ onto some subgroup of $GL_n(R)$ such that
its image under factorization  $R$ by $J$ is $\overline \rho$.

Since the automorphism  $\overline \rho$ identically acts on
matrices with all elements generated by the unit of~$k$, then every
matrix $A\in E(\Phi,R)$ with elements from the subring of~$R$
generated by unit, is mapped under the action of~$\varphi'$ into
some matrix from the set~$A\cdot GL_n(R,J)$.
\end{proof}

Let $a\in E (\Phi,R)$, $a^2=1$. Then the element $e=\frac{1}{2}
(1+a)$ is an idempotent in the ring $M_n(R)$. This idempotent  $e$
defines a decomposition of the free $R$-module $V=R^n$:
$$
V=eV\oplus (1-e)V=V_0\oplus V_1
$$
(the modules $V_0$, $V_1$ are free, since every projective module
over local ring is free). Let $\overline V=\overline V_0 \oplus
\overline V_1$ be decomposition of the $k$-module~$\overline V$ with
respect to~$\overline a$, and
 $\overline e=\frac{1}{2} (1+\overline a)$.

Then we have

\begin{proposition}\label{pr1_1}
The modules \emph{(}subspaces\emph{)}
 $\overline V_0$, $\overline V_1$ are images of the modules $V_0$, $V_1$ under factorization by~$J$.
\end{proposition}
\begin{proof} Let us denote the images of $V_0$, $V_1$ under factorization
by $J$ by $\widetilde V_0$, $\widetilde V_1$, respectively. Since
$V_0=\{ x\in V| ex=x\},$ $V_1= \{ x\in V|ex=0\},$  we have
 $\overline e(\overline x)=\frac{1}{2}(1+\overline a)(\overline x)=\frac{1}{2}
(1+\overline a(\overline
x))=\frac{1}{2}(1+\overline{a(x)})=\overline{e(x)}$. Then
$\widetilde V_0\subseteq \overline V_0$, $\widetilde V_1\subseteq
\overline V_1$.

Let $x=x_0+x_1$, $x_0\in V_0$, $x_1\in V_1$. Then $\overline
e(\overline x)=\overline e(\overline x_0)+\overline e (\overline
x_1)=\overline x_0$. If $\overline x\in \widetilde V_0$, then
$\overline x=\overline x_0$.
\end{proof}

Let $b=\varphi'(a)$. Then $b^2=1$ and $b$ is equivalent to $a$
modulo~$J$.

\begin{proposition}\label{pr1_2}
Suppose that $ a,b\in E_\pi(\Phi,R)$, $a^2=b^2=1$, $a$ is a matrix
with elements from the subring of~$R$, generated by the unit, $b$
and $a$ are equivalent modulo~$J$, $V=V_0\oplus V_1$ is a
decomposition of~$V$ with respect to~$a$, $V=V_0'\oplus V_1'$ is a
decomposition of~$V$ with respect to~$b$. Then $\dim V_0'=\dim V_0$,
$\dim V_1'=\dim V_1$.
\end{proposition}

\begin{proof}
We have an $R$-basis of the module~$V$ $\{ e_1,\dots,e_n\}$ such
that $\{ e_1,\dots,e_k\}\subset V_0$, $\{ e_{k+1},\dots,
e_n\}\subset V_1$.  It is clear that
$$
\overline a \overline e_i=\overline{ae_i}=\overline {(\sum_{j=1}^n
a_{ij} e_j)}= \sum_{j=1}^n \overline a_{ij} \overline e_j.
$$
Let $\overline V=\overline V_0\oplus \overline V_1$, $\overline V =
\overline V_0'\oplus \overline V_1'$ are decompositions of
$k$-module (space)~$\overline V$ with respect to $\overline a$ and $
\overline b$. It is clear that $\overline V_0= \overline V_0'$,
$\overline V_1 =\overline V_1'$. Therefore, by
Proposition~\ref{pr1_1} the images of the modules $V_0$ and $V_0'$,
$V_1$ and $V_1'$ under factorization by~$J$ coincide. Let us take
such $\{ f_1,\dots, f_k\}\subset V_0'$, $\{ f_{k+1},\dots,
f_n\}\subset V_1'$ that $\overline f_i=\overline e_i$,
$i=1,\dots,n$. Since the matrix of transformation from $\{
e_1,\dots, e_n\}$ to $\{ f_1,\dots, f_n\}$ is invertible (it is
equivalent to the identical matrix modulo~$J$) we have that $\{
f_1,\dots, f_n\}$ is a $R$-basis in~$V$. It is clear that $\{
f_1,\dots, f_k\}$ is a $R$-basis in $V_0'$, $\{ v_{k+1},\dots,
v_n\}$ is a $R$-basis in $V_1'$.
\end{proof}

\section{Images of~$w_{\alpha_i}$}

Consider an adjoint Chevalley group $E=E(\Phi,R)$ with one of root
systems $B_2$ or $G_2$, its adjoint representation in the group
$GL_{10}(R)$ (or $GL_{14}(R)$), in the basis of weight vectors
$v_1=x_{\alpha_1}, v_{-1}=x_{-\alpha_1}, \dots, v_l=x_{\alpha_l},
v_{-l}=x_{-\alpha_{-l}}, V_1=h_{1},V_2=h_{2}$, corresponding to the
Chevalley basis of~$B_2$, or $G_2$.

We suppose that with the help of the automorphism~$\varphi$ we
constructed an isomorphism $\varphi'= i_{g^{-1}} \varphi$, described
in the previous section. Recall that it is an isomorphism from
 $E_{ad}(\Phi,R)\subset GL_n(R)$ onto some subgroup in $GL_n(R)$,
such that its image under factorization  $R$ by $J$ coincides with a
ring automorphism $\overline \rho$.

Consider matrices $h_{\alpha_1}(-1), h_{\alpha_2}(-1)$ (see
Definition~14) in our basis. They have the form
\begin{align*}
 h_{\alpha_1}(-1)&=diag [-1,-1,-1,-1,1,1,1,1,1,1],\\
h_{\alpha_2}(-1)&=E
\end{align*}
for $B_2$ and
\begin{align*}
 h_{\alpha_1}(-1)&=diag [1,1,-1,-1,-1,-1,-1,-1,-1,-1,1,1,1,1],\\
h_{\alpha_2}(-1)&=diag[-1,-1,1,1,-1,-1,1,1,-1,-1,-1,-1,1,1]
\end{align*}
for $G_2$.

According to Proposition~\ref{pr1_2} we see that every matrix
$h_i=\varphi'(h_{\alpha_i}(-1))$ in some basis is diagonal with $\pm
1$ on diagonal, the number of  $1$ and $-1$ coincides with this
number for $h_{\alpha_i}(-1)$. Since $h_1$ and $h_2$ commute, there
exists a basis, where $h_1$ and $h_2$ have the same form as
$h_{\alpha_1}(-1)$ and $h_{\alpha_2}(-1)$. Suppose that we come to
this basis with the help of the matrix~$g_1$. It is clear that
 $g_1\in GL_n(R,J)$. Consider the mapping
 $\varphi_1=i_{g_1}^{-1} \varphi'$. It is also an isomorphism
of the group $E$ onto some subgroup of $GL_n(R)$ such that its image
under factorization $R$ by~$J$ is~$\overline \rho$, and
$\varphi_1(h_{\alpha_i}(-1))=h_{\alpha_i}(-1)$ for  $i=1,2$.

Let us consider the isomorphism~$\varphi_1$.

\bigskip

{\bf Remark 1.} Every element $w_i=w_{\alpha_i}(1)$ maps (under
conjugation) diagonal matrices into diagonal matrices, i.\,e., every
image of $w_i$ has block-monomial form. In particular, it can be
written as a block-monomial matrix, where the first block is
$8\times 8$ for $B_2$ and $12\times 12$ for~$G_2$, and the second
block is $2\times 2$.

\bigskip

After conjugation with the matrices from $GL_n(R,J)$ we come from
weight basis to some other basis of~$V$. Consider the first vector
of this new basis, denote it by~$e$. The Weil group $W$ transitively
acts on the roots of the same length, therefore for every
root~$\alpha_i$ of the same length as the first one there exists
such $w^{(\alpha_i)}\in W$ that $w^{(\alpha_i)} \alpha_1=\alpha_i$.
Similarly, all roots of the second length are also conjugate up to
the action of~$W$. Let $\alpha_k$ be the first root of the length
not equal to the length of~$\alpha_1$, and let $f$ be a $k$-th basis
vector after the last basis change. If $\alpha_j$ is a root
conjugate to $\alpha_k$, then denote by $w_{(\alpha_j)}$ the element
of~$W$ such that $w_{(\alpha_j}) \alpha_k=\alpha_j$. onsider now the
basis $e_1,\dots,  e_{2m+2}$, where $e_1=e$, $e_k=f$, and for $1<
i\le 2m$ either $e_i=\varphi_1(w^{(\alpha_i)})e$, or
$e_i=\varphi_1(w_{(\alpha_i)})f$ depending of the length
of~$\alpha_k$; for $2m< i\le 2m+2$  $e_i$ is not changed. It is
clear that the matrix of this basis change is equivalent to $1$
modulo~$J$. Therefore the obtained set of vectors is a basis.

Clear that the matrix $\varphi_1(w_i)$ ($i=1,2$) on the basis part
 $\{ e_1,\dots,e_{12}\}$
coincides with the matrix  $w_i$ in the initial basis of weight
vectors. Since $h_i(-1)$ are squares of $w_i$, their images are not
changed in the new basis.

Moreover, we know (Remark~1), that  $\varphi_1(w_i)$ is
block-diagonal up to the first $8$ ($12$) and last two elements.
Therefore, the last basis part consisting of two elements can be
changed independently.

Let us denote matrices $w_i$ and $\varphi_1(w_i)$ on this part of
basis by $\widetilde w_i$ and $\widetilde{\varphi_1(w_i)}$
respectively, and 2-dimensional module, generated by $e_{2m+1}$ and
$e_{2m+2}$, by $\widetilde V$.

\begin{lemma}\label{l3_3}
For the root systems $B_2$ and $G_2$ there exists such a basis that
$\widetilde{\varphi_1(w_1)}$ and $\widetilde{\varphi_1(w_2)}$ in
this basis have the same form as $\widetilde{w_1}$ and
$\widetilde{w_2}$, i.\,e. are equal to
$$
\begin{pmatrix}
-1& 2\\
0& 1
\end{pmatrix} \text{ and }
\begin{pmatrix}
1& 0\\
1& -1
\end{pmatrix}
$$
for the case $B_2$ and
$$
\begin{pmatrix}
-1& 3\\
0& 1
\end{pmatrix} \text{ and }
\begin{pmatrix}
1& 0\\
1& -1
\end{pmatrix}
$$
for the case $G_2$.
\end{lemma}
\begin{proof}
Since $\widetilde w_1$ is an involution and $\widetilde V_1^1$ has
dimension~$1$, then there exists a basis $\{ e_1,e_2\}$, where
$\widetilde{\varphi_1(w_1)}$ has the form $diag [-1,1]$. In the
basis $\{ e_1, e_2-e_1\}$ the matrix $\widetilde{\varphi_1(w_1)}$
has the obtained form for $B_2$, and in the basis $\{ e_1,
e_2-3/2e_1\}$ it has the obtained form for~$G_2$.

Let the matrix $\widetilde{\varphi_1(w_2)}$ in this basis be
$$
\begin{pmatrix}
a& b\\
c& d
\end{pmatrix}.
$$
Let us make the basis change with the matrix
$$
\begin{pmatrix}
1& (1-a)/c\\
0& 1+\frac{2(1-a)}{c}
\end{pmatrix}.
$$
Under this basis change the matrix $\widetilde{\varphi_1(w_1)}$ is
not moved, and the matrix $\widetilde{\varphi_1(w_2)}$ has the form
$$
\begin{pmatrix}
1& b'\\
c'& d'
\end{pmatrix}.
$$
Since this matrix is an involution, we have $c'(1+d')=0$,
$1+b'c'=1$. Therefore, $d'=-1$, $b'=0$. Now let us consider the
cases $B_2$ and $G_2$ separately.

For the case $B_2$ we can use the condition
$$
\left(\begin{pmatrix} -1& 1\\
0& 1
\end{pmatrix} \begin{pmatrix}
1& 0\\
c'& -1
\end{pmatrix}\right)^2 =
\left(\begin{pmatrix}
1& 0\\
c'& -1
\end{pmatrix}
\begin{pmatrix} -1& 1\\
0& 1
\end{pmatrix}\right)^2.
$$
This condition gives (its second line and first row) $2c'(c'-2)=0$,
so $c'=2$.

For the case $G_2$
$$
\left(\begin{pmatrix} -1& 1\\
0& 1
\end{pmatrix} \begin{pmatrix}
1& 0\\
c'& -1
\end{pmatrix}\right)^3 =
\left(\begin{pmatrix}
1& 0\\
c'& -1
\end{pmatrix}
\begin{pmatrix} -1& 1\\
0& 1
\end{pmatrix}\right)^3,
$$
therefore $3c'(3c'-1)(c'-1)=0$. Since $3\in R^*$, $c'\equiv 1\mod
J$, $3c'-1\equiv 2 \mod J$, we have $c'=1$.
\end{proof}

Thus, now we can move to an isomorphism $\varphi_2$, that is
obtained from $\varphi_1$ by some basis change with the help of a
matrix from $GL_n(R,J)$. It has all described above properties of
$\varphi_1$, but also additionally  $\varphi_2(w_i)=w_i$ for all
$i=1,\dots,l$.

Suppose that we now have namely isomorphism~$\varphi_2$ of this
form.

We need to consider separately the cases $B_2$ and $G_2$. Having and
isomorphism $\varphi_2$, replacing all elements of the Weil group,
we want with one more basis change move to the new isomorphism
$\varphi_3$, that has all properties of $\varphi_2$, but also does
not move all elements $x_{\alpha_i}(1)$, $\alpha_i\in \Phi$.

\section{Images of $x_{\alpha_i}(1)$ in the case $B_2$.}\leavevmode

In the case $B_2$ we have roots (ordered as follows):
$e_1,-e_1,e_2,-e_2, e_1+e_2, -e_1-e_2, e_1-e_2, e_2-e_1$. Therefore,
the adjoint representation has dimension~$10$. In this
representation {\small
$$
w_{e_1-e_2}= \begin{pmatrix}
0& 0& -1& 0& 0& 0& 0& 0& 0& 0\\
0& 0& 0& -1& 0& 0& 0& 0& 0& 0\\
1& 0& 0& 0& 0& 0& 0& 0& 0& 0\\
0& 1& 0& 0& 0& 0& 0& 0& 0& 0\\
0& 0& 0& 0& 1& 0& 0& 0& 0& 0\\
0& 0& 0& 0& 0& 1& 0& 0& 0& 0\\
0& 0& 0& 0& 0& 0& 0& -1& 0& 0\\
0& 0& 0& 0& 0& 0& -1& 0& 0& 0\\
0& 0& 0& 0& 0& 0& 0& 0& -1& 1\\
0& 0& 0& 0& 0& 0& 0& 0& 0& 1
\end{pmatrix},\quad
w_{e_1}= \begin{pmatrix}
0& 1& 0& 0& 0& 0& 0& 0& 0& 0\\
1& 0& 0& 0& 0& 0& 0& 0& 0& 0\\
0& 0& 1& 0& 0& 0& 0& 0& 0& 0\\
0& 0& 0& 1& 0& 0& 0& 0& 0& 0\\
0& 0& 0& 0& 0& 0& 0& -1& 0& 0\\
0& 0& 0& 0& 0& 0& -1& 0& 0& 0\\
0& 0& 0& 0& 0& -1& 0& 0& 0& 0\\
0& 0& 0& 0& -1& 0& 0& 0& 0& 0\\
0& 0& 0& 0& 0& 0& 0& 0& -1& 0\\
0& 0& 0& 0& 0& 0& 0& 0& -2& 1
\end{pmatrix},
$$
$$
w_{e_1+e_2}= \begin{pmatrix}
0& 0& 0& -1& 0& 0& 0& 0& 0& 0\\
0& 0& -1& 0& 0& 0& 0& 0& 0& 0\\
0& 1& 0& 0& 0& 0& 0& 0& 0& 0\\
1& 0& 0& 0& 0& 0& 0& 0& 0& 0\\
0& 0& 0& 0& 0& -1& 0& 0& 0& 0\\
0& 0& 0& 0& -1& 0& 0& 0& 0& 0\\
0& 0& 0& 0& 0& 0& 1& 0& 0& 0\\
0& 0& 0& 0& 0& 0& 0& 1& 0& 0\\
0& 0& 0& 0& 0& 0& 0& 0& 1& -1\\
0& 0& 0& 0& 0& 0& 0& 0& 0& -1
\end{pmatrix},\quad
w_{e_2}= \begin{pmatrix}
-1& 0& 0& 0& 0& 0& 0& 0& 0& 0\\
0& -1& 0& 0& 0& 0& 0& 0& 0& 0\\
0& 0& 0& -1& 0& 0& 0& 0& 0& 0\\
0& 0& -1& 0& 0& 0& 0& 0& 0& 0\\
0& 0& 0& 0& 0& 0& 1& 0& 0& 0\\
0& 0& 0& 0& 0& 0& 0& 1& 0& 0\\
0& 0& 0& 0& 1& 0& 0& 0& 0& 0\\
0& 0& 0& 0& 0& 1& 0& 0& 0& 0\\
0& 0& 0& 0& 0& 0& 0& 0& 1& 0\\
0& 0& 0& 0& 0& 0& 0& 0& 2& -1
\end{pmatrix},
$$
$$
x_{e_2}(1)= \begin{pmatrix}
1& 0& 0& 0& 0& 0& 2& 0& 0& 0\\
0& 1& 0& 0& 0& -2& 0& 0& 0& 0\\
0& 0& 1& -1& 0& 0& 0& 0& 2& -2\\
0& 0& 0& 1& 0& 0& 0& 0& 0& 0\\
1& 0& 0& 0& 1& 0& 1& 0& 0& 0\\
0& 0& 0& 0& 0& 1& 0& 0& 0& 0\\
0& 0& 0& 0& 0& 0& 1& 0& 0& 0\\
0& -1& 0& 0& 0& 1& 0& 1& 0& 0\\
0& 0& 0& 0& 0& 0& 0& 0& 1& 0\\
0& 0& 0& 1& 0& 0& 0& 0& 0& 1
\end{pmatrix}.
$$
}

It is clear that any conditions that hold for elements of the
Chevalley group, hold also and for their images under the
isomorphism~$\varphi_2$. We will use this fact.

 Using the condition $w_{e_1}\cdot x_{e_2}(1)=x_{e_2}(1)\cdot w_{e_1}$,
we obtain {\small
\begin{multline*}
 x_{e_2}=\varphi_2(x_{e_2}(1))=\\
 =\begin{pmatrix}
a_{1,1}& a_{1,2}& a_{1,3}& a_{1,4}& a_{1,5}& a_{1,6}& a_{1,7}& a_{1,8}& a_{1,9}& a_{1,10}\\
a_{1,2}& a_{1,1}& a_{1,3}& a_{1,4}& -a_{1,8}& -a_{1,7}& -a_{1,6}& -a_{1,5}& -a_{1,9}-2a_{1,10}& a_{1,10}\\
a_{3,1}& a_{3,1}& a_{3,3}& a_{3,4}& a_{3,5}& a_{3,6}& -a_{3,6}& -a_{3,5}& a_{3,9}& -a_{3,9}\\
a_{4,1}& a_{4,1}& a_{4,3}& a_{4,4}& a_{4,5}& a_{4,6}& -a_{4,6}& -a_{4,5}& a_{4,9}& -a_{4,9}\\
a_{5,1}& a_{5,2}& a_{5,3}& a_{5,4}& a_{5,5}& a_{5,6}& a_{5,7}& a_{5,8}& a_{5,9}& a_{5,10}\\
a_{6,1}& a_{6,2}& a_{6,3}& a_{6,4}& a_{6,5}& a_{6,6}& a_{6,7}& a_{6,8}& a_{6,9}& a_{6,10}\\
-a_{6,2}& -a_{6,1}& -a_{6,3}& -a_{6,4}& a_{6,8}& a_{6,7}& a_{6,6}& a_{6,5}& a_{6,9}+2a_{6,10}& -a_{6,10}\\
-a_{5,2}& -a_{5,1}& -a_{5,3}& -a_{5,4}& a_{5,8}& a_{5,7}& a_{5,6}& a_{5,5}& a_{5,9}+2a_{5,10}& -a_{5,10}\\
a_{9,1}& -a_{9,1}& 0& 0& a_{9,5}& a_{9,6}& a_{9,6}& a_{9,5}& a_{9,9}& 0\\
a_{10,1}& a_{10,1}-2a_{9,1}& a_{10,3}& a_{10,4}& a_{10,5}& a_{10,6}&
2a_{9,6}-a_{10,6}& 2a_{9,5}-a_{10,5}& a_{10,9}& a_{9,9}-a_{10,9}
\end{pmatrix}
\end{multline*}
} (it can be checked by direct calculations).

 Making basis change with the block-diagonal matrix, that is identical on the basis part
$$\{v_{e_1+e_2}, v_{-e_1-e_2},
v_{e_1-e_2}, v_{e_2-e_1}, V_1, V_2\}$$ and has the form
$$
\begin{pmatrix}
2a_{1,7}/(a_{1,7}^2+a_{1,8}^2)& -2a_{1,8}/(a_{1,7}^2+a_{1,8}^2)&
0&0\\
-2a_{1,8}/(a_{1,7}^2+a_{1,8}^2)& 2a_{1,7}/(a_{1,7}^2+a_{1,8}^2)&
0&0\\
0&0& 2a_{1,7}/(a_{1,7}^2+a_{1,8}^2)& -2a_{1,8}/(a_{1,7}^2+a_{1,8}^2)\\
0& 0& -2a_{1,8}/(a_{1,7}^2+a_{1,8}^2)&
2a_{1,7}/(a_{1,7}^2+a_{1,8}^2)
\end{pmatrix}
$$
on the part $\{ v_{e_1}, v_{-e_1}, v_{e_2}, v_{-e_2}\}$, we do not
move the elements $w_i$, but now $a_{1,7}$ is equal to~$2$, and
$a_{1,8}$ us equal to zero~$0$. At the same time this basis change
is equivalent to identical one modulo~$J$.

Similarly, basis change with the help of block-diagonal matrix, that
is identical on the basis part $\{v_{e_1}, v_{-e_1}, v_{e_2},
v_{-e_2}, V_1, V_2\}$ and having the form
$$
\begin{pmatrix}
2a_{3,9}/(a_{3,9}^2+a_{4,9}^2)& -2a_{4,9}/(a_{3,9}^2+a_{4,9}^2)&
0&0\\
-2a_{4,9}/(a_{3,9}^2+a_{4,9}^2)& 2a_{3,9}/(a_{3,9}^2+a_{4,9}^2)&
0&0\\
0&0& 2a_{3,9}/(a_{3,9}^2+a_{4,9}^2)& -2a_{4,9}/(a_{3,9}^2+a_{4,9}^2)\\
0& 0& -2a_{4,9}/(a_{3,9}^2+a_{4,9}^2)&
2a_{3,9}/(a_{3,9}^2+a_{4,9}^2)
\end{pmatrix}
$$
on the part $\{ v_{e_1+e_2}, v_{-e_1-e_2}, v_{e_1-e_2},
v_{e_2-e_1}\}$, also does not move the elements $w_i$, $a_{1,7}$,
$a_{1,8}$, but now  $a_{3,9}$ is equal to~$2$, $a_{4,9}$ equal
to~$0$. Also this change is identical modulo~$J$.

We suppose that after these two basis changes we move to an
isomorphism~$\varphi_3$.

Now we consider the matrix $x_{e_1+e_2}$, the image of
$x_{e_1+e_2}(1)$. It commutes with $h_{e_1-e_2}(-1)$ and
$w_{e_1-e_2}$, therefore it is equal to
$$
\begin{pmatrix}
b_{1,1}& b_{1,2}& b_{1,3}& b_{1,4}& 0& 0& 0& 0& 0& 0\\
b_{2,1}& b_{2,2}& b_{2,3}& b_{2,4}& 0& 0& 0& 0& 0& 0\\
-b_{1,3}& -b_{1,4}& b_{1,1}& b_{1,2}& 0& 0& 0& 0& 0& 0\\
-b_{2,3}& -b_{2,4}& b_{2,1}& b_{2,2}& 0& 0& 0& 0& 0& 0\\
0& 0& 0& 0& b_{5,5}& b_{5,6}& b_{5,7}& -b_{5,7}& 0& b_{5,10}\\
0& 0& 0& 0& b_{6,5}& b_{6,6}& b_{6,7}& -b_{6,7}& 0& b_{6,10}\\
0& 0& 0& 0& b_{7,5}& b_{7,6}& b_{7,7}& b_{7,8}& b_{7,9}& b_{7,10}\\
0& 0& 0& 0& -b_{7,5}& -b_{7,6}& b_{7,8}& b_{7,7}& b_{7,9}& -b_{7,9}-b_{7,10}\\
0& 0& 0& 0& b_{9,5}& b_{9,6}& b_{9,7}& b_{9,8}& b_{9,9}& b_{9,10}\\
0& 0& 0& 0& 2b_{9,5}& 2b_{9,6}& b_{9,7}-b_{9,8}& b_{9,8}-b_{9,7}& 0&
b_{9,9}+2b_{9,10}
\end{pmatrix}.
$$

We will use the following list of conditions:
\begin{align*}
Con1&:=(x_{e_2}h_{e_1+e_2}(-1)x_{e_2}h_{e_1+e_2}(-1)=1),\\
Con2&:=(x_{e_1+e_2}x_{e_1-e_2}=x_{e_1-e_2}x_{e_1+e_2}):=(x_{e_1+e_2}w_{e_2}x_{e_1+e_2}w_{e_2}^{-1}=w_{e_2}x_{e_1+e_2}w_{e_2}^{-1}x_{e_1+e_2}),\\
Con3&:=(x_{e_1+e_2}x_{e_2}=x_{e_2}x_{e_1+e_2}),\\
Con4&:=(x_{e_1+e_2}^2x_{e_2}x_{e_1}=x_{e_1}x_{e_2}),\\
Con5&:=(x_{e_1-e_2}w_{e_1-e_2}x_{e_1-e_2}w_{e_1-e_2}^{-1}z_{e_1-e_2}=w_{e_1-e_2}).
\end{align*}

Let us denote $y_1=a_{1,1}-1$, $y_{1,2}=a_{1,2}$, $y_3=a_{1,3}$,
$y_4=a_{1,4}$, $y_5=a_{1,5}$, $y_6=a_{1,6}$, $y_7=a_{1,9}$,
$y_8=a_{1,10}$, $y_9=a_{3,1}$, $y_{10}=a_{3,3}-1$,
$y_{11}=a_{3,4}+1$, $y_{12}=a_{3,5}$, $y_{13}=a_{3,6}$,
$y_{14}=a_{4,1}$, $y_{15}=a_{4,3}$, $y_{16}=a_{4,4}-1$,
$y_{17}=a_{4,5}$, $y_{18}=a_{4,6}$, $y_{19}=a_{5,1}-1$,
$y_{20}=a_{5,2}$, $y_{21}=a_{5,3}$, $y_{22}=a_{5,4}$,
$y_{23}=a_{5,5}-1$, $y_{24}=a_{5,6}$, $y_{25}=a_{5,7}-1$,
$y_{26}=a_{5,8}$, $y_{27}=a_{5,9}$, $y_{28}=a_{5,10}$,
$y_{29}=a_{6,1}$, $y_{30}=a_{6,2}$, $y_{31}=a_{6,3}$,
$y_{32}=a_{6,4}$, $y_{33}=a_{6,5}$, $y_{34}=a_{6,6}-1$,
$y_{35}=a_{6,7}$, $y_{36}=a_{6,8}$, $y_{37}=a_{6,9}$,
$y_{38}=a_{6,10}$, $y_{39}a_{9,1}$, $y_{40}=a_{9,5}$,
$y_{41}=a_{9,6}$, $y_{42}=a_{9,9}-1$, $y_{43}=a_{10,1}$,
$y_{44}=a_{10,3}$, $y_{45}=a_{10,4}-1$, $y_{46}=a_{10,5}$,
$y_{47}=a_{10,6}$, $y_{48}=a_{10,9}$, $y_{49}=b_{1,1}-1$,
$y_{50}=b_{1,2}$, $y_{51}=b_{1,3}$, $y_{52}=b_{1,4}+1$,
$y_{53}=b_{2,1}$, $y_{54}=b_{2,2}-1$, $y_{55}=b_{2,3}$,
$y_{56}=b_{2,4}$, $y_{57}=b_{5,5}-1$, $y_{58}=b_{5,6}+1$,
$y_{59}=b_{5,7}$, $y_{60}=b_{5,10}+1$, $y_{61}=b_{6,5}$,
$y_{62}=b_{6,6}-1$, $y_{63}=b_{6,7}$, $y_{64}=b_{6,10}$,
$y_{65}=b_{7,5}$, $y_{66}=b_{7,6}$, $y_{67}=b_{7,7}-1$,
$y_{68}=b_{7,8}$, $y_{69}=b_{7,9}$, $y_{70}=b_{7,10}$,
$y_{71}=b_{9,5}$, $y_{72}=b_{9,6}-1$, $y_{73}=b_{9,7}$,
$y_{74}=b_{9,8}$, $y_{75}=b_{9,9}-1$, $y_{76}=b_{9,10}$. All these
$y_i$ are from~$J$. From conditions 1--5 we can choose $76$
equalities from the following positions (these equations are linear
up to $y_i$):

condition $Con1$: positions $(1,1)$, $(1,2)$, $(1,3)$, $(1,4)$,
$(1,5)$, $(1,6)$, $(1,7)$, $(1,8)$, $(1,9)$, $(1,10)$, $(3,1)$,
$(3,3)$, $(3,4)$, $(3,5)$, $(3,6)$, $(3,10)$, $(4,1)$, $(4,3)$,
$(4,4)$, $(5,1)$, $(5,2)$, $(5,4)$, $(5,6)$, $(5,7)$, $(5,8)$,
$(5,9)$, $(6,6)$, $(9,1)$, $(9,4)$ ($29$ equalities);

condition $Con2$: positions $(1,3)$, $(2,3)$, $(3,3)$, $(5,5)$,
$(5,6)$, $(5,7)$, $(5,8)$, $(5,9)$, $(5,10)$, $(6,8)$, $(9,5)$,
$(9,6)$, $(9,9)$, $(10,5)$, $(10,8)$ ($15$ equalities);

condition $Con3$: positions $(1,1)$, $(1,2)$, $(1,4)$, $(1,6)$,
$(1,7)$, $(2,5)$, $(2,6)$, $(2,9)$, $(2,10)$, $(3,1)$, $(3,2)$,
$(3,4)$, $(3,6)$, $(3,7)$, $(3,10)$, $(5,1)$, $(5,2)$, $(5,4)$,
$(5,5)$, $(5,6)$, $(5,7)$, $(5,9)$, $(8,2)$, $(8,4)$ ($24$
equalities);

condition $Con4$: positions $(5,6)$,  $(5,8)$, $(5,10)$, $(6,6)$,
$(7,6)$, $(10,6)$ ($6$ equalities);

condition $Con5$: positions $(2,4)$, $(6,7)$ ($2$~equalities).

If we write the matrix of this linear system modulo~$J$, we obtain a
matrix $76\times 76$ with entries $0, \pm 1, \pm 2$ and determinant
$2^{36}$ (it is checked by direct calculations). Therefore, its
determinant is invertible in~$R$. Consequently, our system of
equations has a unique solution $y_1=\dots=y_{76}=0$.
 Thus,
$\varphi_3(x_{\alpha_i}(1))=x_{\alpha_i}(1)$.

As for elements $c_t=\varphi_2(h_{e_1+e_2}(t))$, since $c_t$
commutes with $h_{e_1+e_2}(-1)$ and $w_{e_1-e_2}$, we have
$$
c_t=\begin{pmatrix}
c_{1,1}& c_{1,2}& c_{1,3}& c_{1,4}& 0& 0& 0& 0& 0& 0\\
c_{2,1}& c_{2,2}& c_{2,3}& c_{2,4}& 0& 0& 0& 0& 0& 0\\
-c_{1,3}& -c_{1,4}& c_{1,1}& c_{1,2}& 0& 0& 0& 0& 0& 0\\
-c_{2,3}& -c_{2,4}& c_{2,1}& c_{2,2}& 0& 0& 0& 0& 0& 0\\
0& 0& 0& 0& c_{5,5}& c_{5,6}& c_{5,7}& -c_{5,7}& 0& c_{5,10}\\
0& 0& 0& 0& c_{6,5}& c_{6,6}& c_{6,7}& -c_{6,7}& 0& c_{6,10}\\
0& 0& 0& 0& c_{7,5}& c_{7,6}& c_{7,7}& c_{7,8}& c_{7,9}& c_{7,10}\\
0& 0& 0& 0& -c_{7,5}& -c_{7,6}& c_{7,8}& c_{7,7}& c_{7,9}& -c_{7,9}-c_{7,10}\\
0& 0& 0& 0& c_{9,5}& c_{9,6}& c_{9,7}& c_{9,8}& c_{9,9}& c_{9,10}\\
0& 0& 0& 0& 2c_{9,5}& 2c_{9,6}& c_{9,7}-c_{9,8}& c_{9,8}-c_{9,7}& 0&
c_{9,9}+2c_{9,10}
\end{pmatrix}.
$$

From  $c_tx_{e_1-e_2}=x_{e_1-e_2}c_t$ we obtain
$c_{1,2}=c_{1,3}=c_{2,1}=c_{2,4}=c_{5,7}=c_{6,7}=c_{7,5}=c_{7,6}=c_{7,8}=c_{7,9}=c_{9,7}=c_{9,8}=c_{7,10}=0$,
$c_{9,9}=c_{7,7}$.

From  $c_tw_{e_1+e_2}c_t=w_{e_1+e_2}$ it follows $c_{2,3}=-c_{1,4}$,
$c_{7,7}=1$; from
$c_tx_{e_2}c_t^{-1}x_{e_2}-x_{e_2}c_tx_{e_2}c_t^{-1}$ it follows
$c_{1,4}=0$, $c_{1,1}c_{2,2}=1$,
$c_{6,5}=c_{5,6}=c_{5,10}=c_{6,10}=c_{9,5}=c_{9,6}=c_{9,10}=0$,
$c_{5,5}c_{6,6}=1$, and again from the previous condition
$c_{5,5}=c_{1,1}^2$.

So $c_t=h_{e_1+e_2}(s)$ for some $s\in R^*$. Similarly, for
$d_t=\varphi_3(h_{e_2}(t))$ we have $d_t=h_{e_2}(s)$ with the
same~$s$.

Therefore,  $\varphi_3$ is such that
$\varphi_3(x_{\alpha_i}(1))=x_{\alpha_i}(1)$,
$\varphi_3(h_{\alpha_i}(t))=h_{\alpha_i}(s(t))$ for all $\alpha_i\in
\Phi$, $t\in R^*$.

Let us show the same for  $G_2$.

\section{Images of $x_{\alpha_i}(1)$ and
$h_{\alpha_2}(2)$ for the case $G_2$.}\leavevmode

Recall that {\small
$$ w_1=\left(\begin{array}{cccccccccccccc}
0& -1& 0& 0& 0& 0& 0& 0& 0& 0& 0& 0& 0& 0\\
-1& 0& 0& 0& 0& 0& 0& 0& 0& 0& 0& 0& 0& 0\\
0& 0& 0& 0& 0& 0& 0& 0& 1& 0& 0& 0& 0& 0\\
0& 0& 0& 0& 0& 0& 0& 0& 0& 1& 0& 0& 0& 0\\
0& 0& 0& 0& 0& 0& 1& 0& 0& 0& 0& 0& 0& 0\\
0& 0& 0& 0& 0& 0& 0& 1& 0& 0& 0& 0& 0& 0\\
0& 0& 0& 0& -1& 0& 0& 0& 0& 0& 0& 0& 0& 0\\
0& 0& 0& 0& 0& -1& 0& 0& 0& 0& 0& 0& 0& 0\\
0& 0& -1& 0& 0& 0& 0& 0& 0& 0& 0& 0& 0& 0\\
0& 0& 0& -1& 0& 0& 0& 0& 0& 0& 0& 0& 0& 0\\
0& 0& 0& 0& 0& 0& 0& 0& 0& 0& 1& 0& 0& 0\\
0& 0& 0& 0& 0& 0& 0& 0& 0& 0& 0& 1& 0& 0\\
0& 0& 0& 0& 0& 0& 0& 0& 0& 0& 0& 0& -1& 3\\
0& 0& 0& 0& 0& 0& 0& 0& 0& 0& 0& 0& 0& 1
\end{array}\right),
$$
$$
w_2=\left(\begin{array}{cccccccccccccc}
0& 0& 0& 0& 1& 0& 0& 0& 0& 0& 0& 0& 0& 0\\
0& 0& 0& 0& 0& 1& 0& 0& 0& 0& 0& 0& 0& 0\\
0& 0& 0& -1& 0& 0& 0& 0& 0& 0& 0& 0& 0& 0\\
0& 0& -1& 0& 0& 0& 0& 0& 0& 0& 0& 0& 0& 0\\
1& 0& 0& 0& 0& 0& 0& 0& 0& 0& 0& 0& 0& 0\\
0& 1& 0& 0& 0& 0& 0& 0& 0& 0& 0& 0& 0& 0\\
0& 0& 0& 0& 0& 0& 1& 0& 0& 0& 0& 0& 0& 0\\
0& 0& 0& 0& 0& 0& 0& 1& 0& 0& 0& 0& 0& 0\\
0& 0& 0& 0& 0& 0& 0& 0& 0& 0& 1& 0& 0& 0\\
0& 0& 0& 0& 0& 0& 0& 0& 0& 0& 0& 1& 0& 0\\
0& 0& 0& 0& 0& 0& 0& 0& -1& 0& 0& 0& 0& 0\\
0& 0& 0& 0& 0& 0& 0& 0& 0& -1& 0& 0& 0& 0\\
0& 0& 0& 0& 0& 0& 0& 0& 0& 0& 0& 0& 1& 0\\
0& 0& 0& 0& 0& 0& 0& 0& 0& 0& 0& 0& 1& -1
\end{array}\right),
$$
$$
x_{\alpha_1}(1)=\left(\begin{array}{cccccccccccccc}
1& -1& 0& 0& 0& 0& 0& 0& 0& 0& 0& 0& -2& 3\\
0& 1& 0& 0& 0& 0& 0& 0& 0& 0& 0& 0& 0& 0\\
0& 0& 1& 0& 0& 0& 0& 0& 0& 0& 0& 0& 0& 0\\
0& 0& 0& 1& 0& -1& 0& 1& 0& 1& 0& 0& 0& 0\\
0& 0& 3& 0& 1& 0& 0& 0& 0& 0& 0& 0& 0& 0\\
0& 0& 0& 0& 0& 1& 0& -2& 0& -3& 0& 0& 0& 0\\
0& 0& 3& 0& 2& 0& 1& 0& 0& 0& 0& 0& 0& 0\\
0& 0& 0& 0& 0& 0& 0& 1& 0& 3& 0& 0& 0& 0\\
0& 0& -1& 0& -1& 0& -1& 0& 1& 0& 0& 0& 0& 0\\
0& 0& 0& 0& 0& 0& 0& 0& 0& 1& 0& 0& 0& 0\\
0& 0& 0& 0& 0& 0& 0& 0& 0& 0& 1& 0& 0& 0\\
0& 0& 0& 0& 0& 0& 0& 0& 0& 0& 0& 1& 0& 0\\
0& 1& 0& 0& 0& 0& 0& 0& 0& 0& 0& 0& 1& 0\\
0& 0& 0& 0& 0& 0& 0& 0& 0& 0& 0& 0& 0& 1
\end{array}\right),
$$
$$
x_{\alpha_2}(1)= \left(\begin{array}{cccccccccccccc}
1& 0& 0& 0& 0& 0& 0& 0& 0& 0& 0& 0& 0& 0\\
0& 1& 0& 0& 0& 1& 0& 0& 0& 0& 0& 0& 0& 0\\
0& 0& 1& -1& 0& 0& 0& 0& 0& 0& 0& 0& 1& -2\\
0& 0& 0& 1& 0& 0& 0& 0& 0& 0& 0& 0& 0& 0\\
-1& 0& 0& 0& 1& 0& 0& 0& 0& 0& 0& 0& 0& 0\\
0& 0& 0& 0& 0& 1& 0& 0& 0& 0& 0& 0& 0& 0\\
0& 0& 0& 0& 0& 0& 1& 0& 0& 0& 0& 0& 0& 0\\
0& 0& 0& 0& 0& 0& 0& 1& 0& 0& 0& 0& 0& 0\\
0& 0& 0& 0& 0& 0& 0& 0& 1& 0& 0& 0& 0& 0\\
0& 0& 0& 0& 0& 0& 0& 0& 0& 1& 0& 1& 0& 0\\
0& 0& 0& 0& 0& 0& 0& 0& -1& 0& 1& 0& 0& 0\\
0& 0& 0& 0& 0& 0& 0& 0& 0& 0& 0& 1& 0& 0\\
0& 0& 0& 0& 0& 0& 0& 0& 0& 0& 0& 0& 1& 0\\
0& 0& 0& 1& 0& 0& 0& 0& 0& 0& 0& 0& 0& 1
\end{array}\right),
$$
}
$$
h_{\alpha_2}(2)=diag[2,1/2,1/4,4,1/2,2,1,1,2,1/2,1/2,2,1,1].
$$
The fact that $x_1=\varphi_2(x_{\alpha_1}(1))$ commutes with
$h_{\alpha_1}(-1)$ and with
$w_{3\alpha_1+2\alpha_2}=w_2w_1w_2w_1^{-1}w_2^{-1}$, gives
 {\small
$$ x_1=\left(\begin{array}{cccccccccccccc}
a_1& a_2& 0& 0& 0& 0& 0& 0& 0& 0& a_{11}& -a_{11}& a_{13}& -3/2a_{13}\\
b_1& b_2& 0& 0& 0& 0& 0& 0& 0& 0& b_{11}& -b_{11}& b_{13}& -3/2b_{13}\\
0& 0& c_3& c_4& c_5& c_6& c_7& c_8& c_9& c_{10}& 0& 0& 0& 0\\
0& 0& d_3& d_4& d_5& d_6& d_7& d_8& d_9& d_{10}& 0& 0& 0& 0\\
0& 0& e_3& e_4& e_5& e_6& e_7& e_8& e_9& e_{10}& 0& 0& 0& 0\\
0& 0& f_3& f_4& f_5& f_6& f_7& f_8& f_9& f_{10}& 0& 0& 0& 0\\
0& 0& -f_{10}& -f_9& -f_8& -f_7& f_6& f_5& f_4& f_3& 0& 0& 0& 0\\
0& 0& -e_{10}& -e_9& -e_8& -e_7& e_6& e_5& e_4& e_3& 0& 0& 0& 0\\
0& 0& -d_{10}& -d_9& -d_8& -d_7& d_6& d_5& d_4& d_3& 0& 0& 0& 0\\
0& 0& -c_{10}& -c_9& -c_8& -c_7& c_6& c_5& c_4& c_3& 0& 0& 0& 0\\
g_1& g_2& 0& 0& 0& 0& 0& 0& 0& 0& g_{11}& g_{12}& g_{13}& g_{14}\\
-g_1& -g_2& 0& 0& 0& 0& 0& 0& 0& 0& g_{11}& g_{12}& -g_{13}& g_{14}+3g_{13}\\
h_1& h_2& 0& 0& 0& 0& 0& 0& 0& 0& h_{11}& -h_{11}+3i_{11}& h_{13}& 3/2(i_{14}-h_{13})\\
0& 0& 0& 0& 0& 0& 0& 0& 0& 0& i_{11}& i_{11}& 0& i_{14}
\end{array}\right).
$$
}

Similarly, since $x_2=\varphi_2(x_{\alpha_2}(1))$ commutes with
$h_{\alpha_2}(-1)$ and
$w_{2\alpha_1+\alpha_2}=w_1w_2w_1w_2^{-1}w_1^{-1}$, we have
 {\small
$$ x_2=\left(\begin{array}{cccccccccccccc}
j_1& j_2& 0& 0& j_5& j_6& 0& 0& j_9&j_{10}& j_{11}& j_{12}& 0& 0\\
k_1& k_1& 0& 0& k_5& k_6& 0& 0& k_9& k_{10}& k_{11}& k_{12}& 0& 0\\
0& 0& l_3& l_4& 0& 0& l_7& -l_7& 0& 0& 0& 0& l_{13}& -2l_{13}\\
0& 0& m_3& m_4& 0& 0& m_7& -m_7& 0& 0& 0& 0& m_{13}& -2m_{13}\\
-k_6& -k_5& 0& 0& k_2& k_1& 0& 0& -k_{12}& -k_{11}& k_{10}& k_9& 0& 0\\
-j_6& -j_5& 0& 0& j_2& j_1& 0& 0& -j_{12}& -j_{11}& j_{10}& j_9& 0& 0\\
0& 0& n_3& n_4& 0& 0& n_7& n_8& 0& 0& 0& 0& n_{13}& n_{14}\\
0& 0& -n_3& -n_4& 0& 0& n_8& n_7& 0& 0& 0& 0& n_{13}+n_{14}& -n_{14}\\
p_1& p_2& 0& 0& p_5& p_6& 0& 0& p_9& p_{10}& p_{11}& p_{12}& 0& 0\\
q_1& q_2& 0& 0& q_5& q_6& 0& 0& q_9& q_{10}& q_{11}& 0& 0& 0\\
-q_6& -q_5& 0& 0& q_2& q_1& 0& 0& 0& -q_{11}& q_{10}& q_9& 0& 0\\
-p_6& -p_5& 0& 0& p_2& p_1& 0& 0& -p_{12}& -p_{11}& p_{10}& p_9& 0& 0\\
0& 0& 0& 0& 0& 0& 0& s_7+s_8& s_7+s_8& 0& 0& 0& 2s_{13}+s_{14}& 0\\
0& 0& s_3& s_4& 0& 0& s_7& s_8& 0& 0& 0& 0& s_{13}& s_{14}
\end{array}\right).$$
}
 Finally, since $h_{\alpha_2}(2)$ commutes with
$h_{\alpha_1}(-1)$, $h_{\alpha_2}(-1)$, $w_{2\alpha_1+\alpha_2}$,
and also from the equality
$Con6:=(w_2h_{\alpha_2}(2)w_2^{-1}=h_{\alpha_2}(2)^{-1})$, we obtain

{\small
$$
d_2=\varphi_2(h_{\alpha_2}(2))=\left(\begin{array}{cccccccccccccc}
t_1& t_2& 0& 0& 0& 0& 0& 0& 0&0& t_{11}& t_{12}& 0& 0\\
u_1& u_2& 0& 0& 0& 0& 0& 0& 0& 0& u_{11}& u_{12}& 0& 0\\
0& 0& v_3& v_4& 0& 0& v_7& -v_7& 0& 0& 0& 0& 0& 0\\
0& 0& w_3& w_4& 0& 0& w_7& -w_7& 0& 0& 0& 0& 0& 0\\
0& 0& 0& 0& u_2& u_1& 0& 0& -u_{12}& -u_{11}& 0& 0& 0& 0\\
0& 0& 0& 0& t_2& t_1& 0& 0& -t_{12}& -t_{11}& 0& 0& 0& 0\\
0& 0& x_3& x_4& 0& 0& x_7& x_8& 0& 0& 0& 0& 0& 0\\
0& 0& -x_3& -x_4& 0& 0& x_8& x_7& 0& 0& 0& 0& 0& 0\\
0& 0& 0& 0& y_5& y_6& 0& 0& y_9& y_{10}& 0& 0& 0& 0\\
0& 0& 0& 0& z_5& z_6& 0& 0& z_9& z_{10}& 0& 0& 0& 0\\
-z_6& -z_5& 0& 0& 0& 0& 0& 0& 0& 0& z_{10}& z_9& 0& 0\\
-y_6& -y_5& 0& 0& 0& 0& 0& 0& 0& 0& y_{10}& y_9& 0& 0\\
0& 0& 0& 0& 0& 0& 0& 0& 0& 0& 0& 0& 1& 0\\
0& 0& 0& 0& 0& 0& 0& 0& 0& 0& 0& 0& 0& 1
\end{array}\right).$$
}

 Pos. $(14,12)$ of $Con7:=(x_1^2d_2=d_2x_1)$ gives
 $$
i_{11}\cdot \alpha=i_{11}((g_{11}+g_{12}+i_{14})(z_9+y_9)-1)=0.
$$
Since $\alpha\equiv 3\mod J$ and $3\in R^*$, we have $i_{11}=0$.
From the position $(14,14)$ of the same condition we obtain
$i_{14}(i_{14}-1)=0$, therefore, $i_{14}=1$.

Now we will use the condition
$$
Con8:=(w_2x_2w_2^{-1}x_1=x_1w_2x_2w_2^{-1}).
$$
Its position $(14,2)$ gives $s_{13}=0$,

Let us make the block-diagonal basis change that is identical on the
submodule, generated by all $\{ V_1,\dots, V_l\}$, and all $\{
v_i,v_{-i}\}$, where $\alpha_i$ is a long root, and with the matrix
$$
\begin{pmatrix}
\frac{a_{13}}{a_{13}^2-b_{13}^2}& \frac{-b_{13}}{a_{13}^2-b_{13}^2}\\
\frac{-b_{13}}{a_{13}^2-b_{13}^2}& \frac{a_{13}}{a_{13}^2-b_{13}^2}
\end{pmatrix}
$$
on all submodules, generated by $\{ v_i,v_{-i}\}$, were $\alpha_i$
is a short root. This basis change does not move $w_i$ for all $i$,
it is equivalent to unit modulo~$J$. At the same time we have
$a_{13}=-2$, $b_{13}=0$.

The position $(13,13)$ of
$$
Con9:=(h_{\alpha_2}(-1)x_1h_{\alpha_2}(-1)x_1=1)
$$
gives
$$
2h_1-2g_{13}h_{11}+h_{13}^2-1=0,
$$
and the position $(13,13)$ of $Con7$ gives $$
-2h_1+2g_{13}h_{11}+h_{13}^2-h_{13}=0.
$$
These two equalities imply $h_{13}=1$.

Now the positions $(2,12)$ and $(2,13)$ of $Con9$ give
$$
\begin{cases}
-b_1a_{11}+b_{11}(-b_2+g_{12}-g_{11})=0,\\
-b_12+b_{11}g_{13}=0.
\end{cases}
$$
This system modulo~$J$ is equivalent to
$$
\begin{cases}
0\cdot \overline b_1-2\cdot \overline b_{11}=0,\\
-2\cdot \overline b_1+ 0\cdot \overline b_{11}=0,
\end{cases}
$$
therefore $b_1=b_{11}=0$. The same condition (pos.~$(2,2)$) directly
implies $b_2=1$.

Let us again make the block-diagonal basis change, that is identical
on the submodule, generated by all $\{ V_1,\dots, V_l\}$ and all $\{
v_i,v_{-i}\}$, where $\alpha_i$ is a short root, and has the matrix
$$
\begin{pmatrix}
\frac{l_{13}}{l_{13}^2-m_{13}^2}& \frac{-m_{13}}{l_{13}^2-m_{13}^2}\\
\frac{-m_{13}}{l_{13}^2-m_{13}^2}& \frac{l_{13}}{l_{13}^2-m_{13}^2}
\end{pmatrix}
$$
on all submodules, generated by $\{ v_i,v_{-i}\}$, where $\alpha_i$
is a long root. This basis change does not move $w_i$ for all $i$,
and is equivalent to~$1$ modulo~$J$. At the same time we have
$l_{13}=1$, $m_{13}=0$.

Now we suppose that from the isomorphism $\varphi_2$ after the last
two basis changes we come to the isomorphism~$\varphi_3$.

From the position   $(13,2)$ of $Con9$ we obtain
$h_1a_2=-2h_{11}g_2$, therefore,  $h_1=\beta_1 h_{11}$,
$\beta_1\equiv 0\mod J$. Now from $(1,12)$ of  $Con9$
$-a_1a_{11}+a_{11}g_{12}-a_{11}g_{11}-2h_{11}$, consequently,
$a_{11}=\beta_2 h_{11}$, $\beta_2\equiv -1 \mod J$. Similarly, from
$(1,13)$ of $Con9$ it follows $a_1=1+\beta_3 h_{11}$, $\beta_3\equiv
0\mod J$, from $(12,1)$ of $Con9$ $g_{12}=\beta_4 h_{11}$, it
follows $\beta_4\equiv 0\mod J$, from $(12,12)$ of $Con9$ it follows
$g_{11}= 1+\beta_5 h_{11}$, $\beta_5\equiv 0\mod J$, from $(12,1)$
of $Con9$ it follows $g_1=beta_6h_{11}$, $\beta_6\equiv 0\mod J$.
Using the position $(2,1)$ of
$$
Con10:=(x_1w_1x_1w_1^{-1}=w_1),
$$
we obtain $a_2=-1+\beta_7h_{11}$, $\beta_7\equiv 0\mod J$. From
$(1,2)$ of $Con9$ it follows $h_2=1+\beta_8 h_{11}$, $\beta_8\equiv
0\mod J$. Using position $(1,2)$ of
$$
Con11:=( w_2x_2w_2^{-1}x_1=x_1w_2x_2w_2^{-1}), $$ we obtain $
j_2=\beta_9 h_{11}$, $\beta_9\equiv 0\mod J$. From $(13,12)$ of
$Con11$ $j_{11}=\beta_{10}h_{11}$, $\beta_{10}\equiv  0\mod J$, from
$(1,11)$ of $Con11$ $j_{12}=\beta_{11}h_{11}$, $\beta_{11}\equiv
0\mod J$, from $(1,1)$ of $Con7$ $u_1=\beta_{12}h_{11}$,
$\beta_{12}\equiv 0\mod J$, from $(1,12)$ of $Con7$
$u_{12}=\beta_{13}h_{11}$, $\beta_{13}\equiv 3/2\mod J$, from
$(13,2)$ of $Con7$ $u_2=1/2+\beta_{13}h_{11}$, $\beta_{13}\equiv
0\mod J$.

Using the position $(7,7)$ of $$ Con12:= (d_2
\varphi_2(x_{2\alpha_1+\alpha_2}(1))=\varphi_2(x_{2\alpha_1+\alpha_2}(1))d_2=d_2w_1w_2x_1w_2^{-1}w_1^{-1}-w_1w_2x_1w_2^{-1}w_1^{-1}d_2)
$$ we obtain $x_8=\beta_{14}h_{11}$, $\beta_{14}\equiv 0\mod J$,
from the position $(14,8)$ of $Con12$ it follows
$x_7=1+\beta_{15}h_{11}$, $\beta_{15}\equiv 0\mod J$, from $(1,11)$
of $Con7$ it follows $u_{11}=\beta_{15}h_{11}$, $\beta_{15}\equiv
0\mod J$.

Using the position $(1,12)$ of $$ Con13:=(w_2d_2w_2^{-1}d_2=1),$$ we
obtain $t_{12}=\beta_{16}h_{11}$, $\beta_{16}\equiv 0\mod J$, from
$(1,11)$ of $Con13$ it follows $t_{11}=\beta_{17}h_{11}$,
$\beta_{17}\equiv -3/2\mod J$, from $(1,5)$ of $Con13$ it follows
$t_2=\beta_{18}h_{11}$, $\beta_{18}\equiv 0\mod J$, from $(2,2)$ of
$Con13$ it follows $t_1=2+\beta_{19}h_{11}$, $\beta_{19}\equiv 0\mod
J$, from $(7,4)$ of $Con12$ it follows $x_4=\beta_{20}h_{11}$,
$\beta_{20}\equiv 0\mod J$, from $(7,3)$ of $Con12$ it follows
$x_3=\beta_{21}h_{11}$, $\beta_{21}\equiv 0\mod J$.

Position  $(13,14)$ of $$
Con14:=(x_2\varphi_2(x_{2\alpha_1+\alpha_2}(1))=\varphi_2(x_{2\alpha_1+\alpha_2}(1))x_2)
$$
gives $n_{14}=\beta_{22}h_{11}$, $\beta_{22}\equiv -2\mod J$, from $
(13,3)$ of $Con14$  it follows $n_3=\beta_{23}h_{11}$,
$\beta_{23}\equiv 0\mod J$, from $(13,4)$ of $Con14$ it follows
$n_4=\beta_{24}h_{11}$, $\beta_{24}\equiv 0\mod J$, from $(8,13)$ of
$Con14$ it follows $n_8=\beta_{25}h_{11}$, $\beta_{25}\equiv 0\mod
J$.

Finally, from $(7,4)$ of $Con14$ it follows $\beta_{26}h_{11}=0$,
where $\beta_{26}\in R^*$. Consequently, $h_{11}=0$.

Now position $(3,14)$ of $Con11$ gives $c_4=0$, position $(4,14)$ of
the same condition gives $d_4=1$, position $(13,2)$ gives
$s_{14}=j_1$.

Position $(3,13)$ of
$$
Con15:=(h_{\alpha_1}(-1) x_2h_{\alpha_1}(-1)x_2=1)
$$
gives $l_3=j_1$, position $(4,3)$ gives $m_3(m_4+1)=0 \Rightarrow
m_3=0$, position $(3,3)$ gives  $s_3=0$, $(4,4)$ gives $m_4=1$,
$(4,7)$ gives $m_7(1+n_7)=0\Rightarrow m_7=0$, $(3,13)$ gives
$j_1=1$, $(3,4)$ gives $s_4=-l_4$.

Position (5,3) of $Con11$ gives $e_4l_4=0\Rightarrow e_4=0$,
position $(4,13)$ of the same condition gives
$(d_7+d_8)n_{13}=0\Rightarrow n_{13}=0$, position $(8,8)$ of $Con15$
gives $n_7=1$, position $(7,13)$ of $Con11$ gives $f_9=0$, position
$(8,13)$ gives $e_9=0$, position $(13,13)$ of $Con14$ gives
$s_8=-s_7$, position $(14,13)$ gives $-l_4(-2g_{13}-g_{14})=s_7$,
and position $(3,13)$ gives $l_4(-2g_{13}-g_{14})=l_7$, therefore
$s_7=-l_7$. Position $(13,5)$ of $Con11$ gives $j_6=0$, $(13,6)$
gives $j_5=0$, $(13,9)$ gives $j_{10}=0$, $(13,10)$ gives $j_9=0$.

Positions $(10,13)$ and $(10,14)$ of

\begin{multline*}
 Con16:=
(\varphi_3(x_{3\alpha_1+\alpha_2}(1)x_{\alpha_2}(1)x_{3\alpha_1+2\alpha_2}(1))
=\varphi_3(x_{\alpha_2}(1)x_{3\alpha_1+\alpha_2}(1))):=\\
=(w_1x_2w_1^{-1}x_2w_2w_1x_2w_2^{-1}w_2^{-1}=w_1x_2w_1^{-1}x_2)
\end{multline*}
 give $q_9=q_{11}=0$,  $(7,13)$ --- $k_{12}=0$,
 $(9,13)$ --- $p_9=1$, $(2,13)$ --- $k_9=0$, $(2,14)$ ---
$k_{11}=0$, $(5,14)$ --- $k_{10}=0$, $(14,12)$ --- $q_{12}=1$,
$(9,14)$ --- $p_{11}=0$, $(4,14)$ --- $p_{10}q_{12}=0\Rightarrow
p_{10}=0$, $(12,14)$ --- $p_{12}=0$, $(11,14)$ --- $q_{10}=1$,
$(3,10)$ --- $l_4=-1$, $(5,1)$ --- $k_6(k_2-1)=0\Rightarrow k_2=1$,
$(5,2)$ --- $k_6k_1=0\Rightarrow k_1=0$, $(4,1)$ и $(4,2)$ ---
$p_5=p_6=0$, $(13,5)$ --- $q_5=0$, $(4,6)$ --- $q_1k_6=0\Rightarrow
q_1=0$, $(1,6)$ --- $k_5k_6=0\Rightarrow k_5=0$, $(5,8)$ ---
$k_6(k_6-1)=0\Rightarrow k_6=1$, $(12,8)$ --- $p_2=0$, $(13,6)$ ---
$q_6=0$.

From $Con11$ it now follows that $e_3=3$ (pos.~(1,3)), $f_4=0$
(pos.~(6,3)), $d_9=0$ (pos.~(9,3)), $c_9=0$ (pos.(10,3)), $e_6=0$
(pos.~(1,6)), $e_{10}=0$ (pos.~(1,10)), $e_5=1$ (pos.~(1,5)).

From $Con14$ it follows $f_7=0$ (pos.~(1,1)), $f_3=0$ (pos.~(1,12)).
Position $(1,3)$ of $Con7$ gives $w_3=0$, positions (3,4) and (4,3)
of $Con9$ give  $v_4=0$.

From $(12,2)$ of $Con9$ we obtain $g_{13}=2g_2$, from $(8,3)$ we
obtain $e_8=-c_{10}$, from $(12,14)$ of $Con10$ we obtain
$g_{14}=-3g_2$, from $(6,9)$ we have $f_{10}(d_3+f_5)=0\Rightarrow
f_5=-d_3$, from $(10,6)$ of $Con15$ --- $q_2=-p_1$, from $(1,7)$ of
$Con11$ we have $e_7=3l_7$, from $(1,8)$ we obtain $c_{10}=3l_7$,
from $(12,3)$ we have $g_2=l_7$, from $(11,3)$ it follows
$p_1=-l_7$.

Now from
$$ Con17:=(d_2x_2^4=x_2d_2)
$$
$v_4=16v_3$ (pos.~(3,4)), $v_3=1/4$ (pos.~(3,13)), $w_7=-3l_7$
(pos.~(3,8)), $y_{10}=0$ (pos.(9,12)), $y_9=4z_{10}$ (pos.~(10,12)),
$z_{10}=x_{10}$ (pos.~(11,9)). From (11,12) and (12,12) of  $Con6$
we have $z_9=0$, $z_{10}=1/2$, from (9,5) and (9,6)  $z_6=-y_5$,
$y_6=-4z_5$, from (4,8)  $v_7=-3/4l_7$. Again from $Con17$ it
follows $z_5=0$ (pos.~(11,6)), $y_5=-3/2l_7$ (pos.~(10,2)).  From
(3,2) of $Con11$ we obtain $c_6=-3l_7^2$, from (4,3) we have
$c_3=1+l_7f_{10}$, from (4,2) we have $d_6=-1-d_{10}l_7$, from (6,2)
we have $f_6=1-l_7f_{10}$, from (4,8) we have $c_8=l_7d_3$, from
(12,6) $c_7=l_7f_{10}$, from (9,12)  $d_3=-l_7f_{10}$. From (5,7) of
$Con7$ we obtain $3l_7(l_7f_{10}-f_{10})=0$. Since $f_{10}\in R^*$,
we have $l_7=0$. From $Con11$ it follows $c_5=0$, from (4,7) of
$Con9$ it follows $d_7=0$, from (4,3) it follows $d_5=0$, from (3,9)
of $Con10$ it follows $d_{10}=1$, from $(5,9)$ we have $f_{10}=-3$,
from (3,7) we have $d_8=1$, from (3,5) we have $f_8=-2$.

Now, finally, we see that $x_1=x_{\alpha_1}(1)$,
$x_2=x_{\alpha_2}(1)$, $d_2=h_{\alpha_2}(2)$. It is easy to check
that $\varphi_3(h_{\alpha_1}(t))=h_{\alpha_1}(s)$ and
$\varphi_3(h_{\alpha_2}(t))=h_{\alpha_2}(s)$ for some $s\in R^*$.
Since all roots of the same length are conjugate up to the action
of~$W$, then $\varphi_3(x_{\alpha_i}(1))=x_{\alpha_i}(1)$ and
$\varphi_3(h_{\alpha_i}(t))=h_{\alpha_i}(s)$ for some~$s\in R^*$.

\section{Proof of Theorem 1.}

Now we have stated that for both root systems under consideration
$\varphi_3(x_{\alpha_i}(1))=x_{\alpha_i}(1))$,
$\varphi_3(h_{\alpha_i}(t))=h_{\alpha_i}(s)$, $i=\pm 1,\dots,\pm m$,
$t,s\in R^*$.

For any long root $\alpha_j$ there exists a root $\alpha_k$ such
that
$h_{\alpha_k}(t)x_{\alpha_j}(1)h_{\alpha_k}(t)^{-1}=x_{\alpha_j}(t)$.
Therefore, $\varphi_3(x_{\alpha_j}(t)) =x_{\alpha_j}(s)$. From these
conditions and commutator conditions it follows that
$\varphi_3(x_{\alpha_j}(t))=x_{\alpha_j}(s)$ for all $\alpha_j\in
\Phi$.

Let us denote the mapping $t\mapsto s$ by $\rho: R^* \to R^*$. If
$t\notin R^*$, then $t\in J$, i.\,e., $t=1+t_1$, where $t_1\in R^*$.
Then
$\varphi_3(x_\alpha(t))=\varphi_3(x_\alpha(1)x_\alpha(t_1))=x_\alpha(1)x_\alpha(\rho(t_1))=
x_\alpha(1+\rho(t_1))$, $\alpha\in \Phi$. Therefore, if we extend
the mapping $\rho$ to the whole ring~$R$ (with the formula
$\rho(t):=1+\rho(t-1)$ for $t\in J$), then we obtain
$\varphi_3(x_\alpha(t))=x_\alpha(\rho(t))$ for all $t\in R$,
$\alpha\in \Phi$.
 It is clear that $\rho$ is injective, additive, multiplicative on invertible elements. Since every element of the ring~$R$ is a sum of two invertible elements,
we have that $\rho$ is also multiplicative on uninvertible elements
of the ring, i.e., is an isomorphism from~$R$ to some its
subring~$R'$.
 Note that in this situation $C E(\Phi,R)
C^{-1}=E(\Phi,R')$ for some matrix $C\in GL(V)$. Let us show that
$R'=R$.

\begin{lemma}\label{porozhd}
Elementary Chevalley group $E_{ad}(R,\Phi)$ generates  $M_n(R)$ as a
ring.
\end{lemma}
\begin{proof}
Let us consider the case of the root system $B_2$, since the case
$G_2$ is completely similar. The matrix $(x_{e_1+e_2}(1)-1)^2$ is
$-2E_{5,6}$ ($E_{ij}$ is a matrix unit). Multiplying it to some
appropriate diagonal matrix, we obtain an arbitrary matrix of the
form $\alpha\cdot e_{12}$ (since $-2\in R^*$ and invertible elements
of~$R$ generate~$R$). Then $w_{e_1}\alpha E_{5,6}=\alpha E_{8,5}$,
$w_{e_1}\alpha E_{5,6} w_{e_1}=\alpha E_{8,7}$, $w_{e_2}\alpha
E_{5,6}=\alpha E_{7,6}$, $w_{e_2}\alpha E_{5,6}w_{e_1}=\alpha
E_{7,7}$, $w_{e_2} \alpha E_{5,6} w_{e_2}=\alpha E_{7,8}$, $w_{e_1}
\alpha E_{5,6} w_{e_2}=\alpha E_{8,8}$, $w_{e_1+e_2} \alpha
E_{5,6}=\alpha E_{6,6}$. These matrices generate a subring of the
matrix ring, generated by $E_{i,j}$, $4\le i,j\le 8$. Similarly,
with the help of $(x_{e_1}(1)-1)^2$ we can obtain a subring,
generated by $E_{i,j}$, $1\le i,j\le 4$. With these matrix units and
elements   $x_\alpha (1)$ we can generate the subring $M_{8}(R)$.
Now let us subtract from $x_{e_2}(1)-1$ appropriate matrix units,
and we obtain the matrix $E_{10, 4}-2 E_{3,9}+ E_{3,10}$.
Multiplying it (from the right side) to $E_{4,i}$, $1\le i\le 8$, we
obtain all $E_{10, i}$, $1\le i\le 2m$. With the help of the Weil
group we have all $E_{i,j}$, $8< i\le 9$, $1\le j\le 8$. Now we have
the matrix $-2E_{3,9}+E_{3,10}$. Multiplying it (from the left side)
to $E_{2m+1,1}$, we have $E_{2m+1,2m+1}$. With the help of last two
matrices we have $E_{3,9}$, therefore
 $E_{i,j}$, $1\le i\le 8$, $8< j\le 9$.
Thus we obtain all matrix units, and so the whole matrix ring
$M_n(R)$.
\end{proof}

\begin{lemma}\label{Tema}
If for some $C\in GL(V)$ we have $C E(\Phi,R) C^{-1}= E(R',\Phi)$,
where $R'$ is a subring in~$R$, then $R'=R$.
\end{lemma}
\begin{proof}
Suppose that $R'$ is a proper subring in~$R$.

Then $C M_n(R) C^{-1} =M_n (R')$, since the group $E(\Phi,R)$
generates the ring $M_n(R)$, and the group $E(\Phi,R')=CE(\Phi,R)
C^{-1}$ generates the ring $M_n(R')$. It is impossible, since $C\in
GL_n(R)$.
\end{proof}

Consequently, we have proved that $\rho$ is an automorphism of the
ring~$R$. So the composition of the initial isomorphism~$\varphi'$
and some basis change with a matrix $C\in GL_n(R)$ (that maps
$E(\Phi,R)$ onto itself), is a ring automorphism~$\rho$. Thus
$\varphi'' = i_{C^{-1}} \circ \rho$.

Therefore, Theorem 1 is proved.

\section{Proof of Theorem 2.}

{\bf Proof of Theorem 2.} Suppose that we have a Chevalley group
$G_{ad}(\Phi,R)$ and its automorphism~$\varphi$. Since the
elementary subgroup $E_{ad}(\Phi,R)$ is the commutant of
$G_{ad}(\Phi,R)$, then it is mapped onto itself under the action
of~$\varphi$. On the elementary subgroup $\varphi$ is a composition
of standard automorphisms: $\varphi=i_C \circ \rho$. Consider a
mapping $\varphi'= \rho^{-1} \circ i_C^{-1}\circ \varphi$. It is an
isomorphism from the Chevalley group $G_{ad}(\Phi,R)$ onto some
subgroup~$G\subset GL_n(R)$, that it identical on the elementary
subgroup. We know that $G_{ad}(\Phi,R)=E_{ad}(\Phi,R)
T_{ad}(\Phi,R)$, $T_{ad}(\Phi,R)$ consists of such elements
$h_{\chi}$, that $\chi: \Phi\to R^*$ is a homomorphism, $h_{\chi}
x_\alpha(t) h_{\chi}^{-1}= x_{\alpha}(\chi(\alpha)\cdot t)$. Every
$h_\chi$ commutes with all $h_\alpha (t)$, $t\in R^*$.

Consider a matrix $A=\varphi'(h_\chi)$. Since  $A$ commutes with all
$h_\alpha (t)$, $\alpha\in \Phi$, $t\in R^*$, it has the form
$$
A=\begin{pmatrix}
D& 0\\
0& C
\end{pmatrix},
$$
where $D$ is a diagonal matrix  $(2m)\times (2m)$, $C$ is some
matrix $l\times l$. For all $\alpha \in \Phi$ and $t\in R$ we have
the condition
$$
A x_\alpha (t) A^{-1}= x_\alpha (\chi(\alpha)\cdot t).
$$
By direct calculations  (as it was done in the previous sections) we
see that this condition implies $A= a_{\chi} h_\chi$, where
$a_\chi\in R^*$. Let $h_{t_1,\dots,t_l}$ be a homomorphism from the
root lattice to $R^*$ such that every simple root  $\alpha_i$ is
mapped to $t_i$. Then from $(h_{t,1,\dots,1})^2\cdot
(h_{1,t^{-1},1,\dots,1})=h_{\alpha_1}(t)$ it follows that for all
$\chi$ $a_\chi^3=1$. Since $\varphi'(h_\chi)\in SL_n(R)$, we have
$a_\chi^n=1$. So (since $10,14$, that are dimensions of adjoint
representations of Chevalley groups $B_2$ and $G_2$, do not divided
to~$3$), we have $a_\chi=1$ for all~$\chi$, therefore, $\varphi'$ is
identical. Theorem is proved.

\begin{corollary}
Any automorphism of an (elementary) adjoint Chevalley group of types
$B_2$ (or $G_2$) over local commutative ring with $1/2$ ($1/6$)
induces an automorphism of the matrix ring $M_n(R)$, where $n$ is a
dimension of adjoint representation.
\end{corollary}

\end{document}